\documentclass[10pt]{amsart}

\usepackage{hyperref}
\hypersetup{colorlinks,citecolor=blue,linkcolor=blue}
\usepackage{amsfonts, amssymb, wasysym}
\usepackage{multicol}
\usepackage{amsmath}
\usepackage{latexsym}
\usepackage{mathrsfs}
\usepackage{epsfig}
\usepackage{graphicx}
\usepackage{tikz}
\usetikzlibrary{positioning}

\numberwithin{equation}{section}
\newtheorem{theorem}{Theorem}[section]

\newtheorem{lemma}[theorem]{Lemma}
\newtheorem{corollary}[theorem]{Corollary}
\newtheorem{conjecture}[theorem]{Conjecture}
\usepackage{cyr}

\DeclareSymbolFont{mcyr}{OT1}{wncyr}{m}{n}
\DeclareMathSymbol\Lob{\mathop}{mcyr}{76}
\newcommand{\Z}{{\mathbb Z}}
\newcommand{\R}{{\mathbb R}}
\newcommand{\Q}{{\mathbb Q}}
\newcommand{\C}{{\mathbb C}}
\newcommand{\A}{{\mathbb A}}

\newcommand{\SO}{\mathrm{SO}}

\newcommand{\PO}{\mathrm{PO}}

\newcommand{\PGO}{\mathrm{PGO}}

\newcommand{\Isom}{\mathrm{Isom}}
\newcommand{\Gal}{\mathrm{Gal}}

\renewcommand{\o}{{\mathfrak o}}
\renewcommand{\O}{{\mathrm O}}

\newcommand{\integers}{{\mathbb Z}}
\newcommand{\realnos}{{\mathbb R}}
\newcommand{\rationalnos}{{\mathbb Q}}
\newcommand{\complexnos}{{\mathbb C}}

\begin{document}

\title{Salem numbers and arithmetic hyperbolic groups}

\author{Vincent Emery, John G. Ratcliffe and Steven T. Tschantz}

%\thank{The first author is supported by the SNSF, project no PP00P2\_157583.}

\address{Mathematisches Institut, University of Bern, Sidlerstrasse 5, 3012 Bern, Switzerland}
\address{Department of Mathematics, Vanderbilt University, Nashville, TN 37240
\vspace{.1in}}

\email{vincent.emery@math.ch}
\email{j.g.ratcliffe@vanderbilt.edu}
\email{steven.tschantz@vanderbilt.edu}


\date{}


\begin{abstract}
In this paper we prove that there is a direct relationship between Salem numbers and translation lengths 
of hyperbolic elements of arithmetic hyperbolic groups that are determined by a quadratic form 
over a totally real number field. 
As an application we determine a sharp lower bound for the length of a closed geodesic in a noncompact 
arithmetic hyperbolic $n$-orbifold for each dimension $n$. 
We also discuss a ``short geodesic conjecture'', and prove its
equivalence with ``Lehmer's conjecture'' for Salem numbers. 
\end{abstract}

\maketitle

\section{Introduction} %1
\label{sec:1}

\subsection{Salem numbers and translation lengths}

In this paper, a {\it Salem number} is a real algebraic integer $\lambda > 1$ that is
conjugate to $\lambda^{-1}$ and whose remaining conjugates lie on the unit circle. We
denote by $\deg\lambda$ the degree of the minimal polynomial of $\lambda$.
Note that we allow $\deg\lambda = 2$ (which occurs exactly when $\lambda+\lambda^{-1}\in
\integers$).

Salem numbers occur in many areas of mathematics (see the surveys \cite{G-H} and \cite{Smyth}). 
In this paper, we show that Salem numbers are directly related to the translation lengths 
of hyperbolic elements of an arithmetic hyperbolic group of the simplest type. 
Our main results are Theorems \ref{thm:1} and \ref{thm:2} below which sharpen and generalize to all dimensions  
results obtained by T. Chinburg, W. Neumann and A. Reid  in dimension 2 and 3 (see
\cite[\S\,4]{N-R}).

Let $H^n$ be the hyperbolic $n$-space, and $\Isom(H^n)$
its group of isometries.
An element $\gamma \in \Isom(H^n)$ is {\it hyperbolic} if there is a unique geodesic 
$L$ in $H^n$, called the {\it axis} of $\gamma$, along which $\gamma$ acts as a translation
by a positive distance $\ell(\gamma)$ called the {\it translation length} of $\gamma$. 
If $\Gamma \subseteq {\rm Isom}(H^n)$ is a lattice, 
then most of the elements of $\Gamma$ are hyperbolic. 
Among arithmetic lattices $\Gamma \subseteq\Isom(H^n)$, those defined  in terms of an admissible quadratic form 
over a totally real number field $K$ are said to be \emph{of the simplest type} (see
\S\,\ref{sec:groups-simplest}).  In this case, $\Gamma$ is \emph{defined over $K$}.

We introduce the following notation: $\Gamma^{(2)}$ is the subgroup of $\Gamma$ generated
by the squares of elements of $\Gamma$ (it is of finite index in $\Gamma$). For $K \subseteq\C$ a
number field and $\lambda \in \C$ an algebraic number, $\deg_K(\lambda)$ will denote the degree
$[K(\lambda):K]$. In particular, we have $\deg_\Q(\lambda) = \deg\lambda$. 

\begin{theorem}  % 1 ( 1.1 )
\label{thm:1}
Let $\Gamma \subseteq \Isom(H^n)$ be an arithmetic lattice   
of the simplest type defined over a totally real number field $K$.
Let $\gamma$ be a hyperbolic element of $\Gamma$, and let $\lambda = e^{\ell(\gamma)}$.  
If $n$ is even or $\gamma \in \Gamma^{(2)}$, then $\lambda$ is a Salem number 
such that $K \subseteq \rationalnos(\lambda +\lambda^{-1})$ and $\deg_K(\lambda) \leq n+1$. 

Conversely, if $\lambda$ is a Salem number, 
and $K$ is a subfield of $\rationalnos(\lambda +\lambda^{-1})$ such that $\deg_K(\lambda) \leq n+1$, 
then there exist an arithmetic lattice $\Gamma \subseteq\Isom(H^n)$ of the simplest type defined 
over $K$ and a hyperbolic element $\gamma$ in $\Gamma$ such that $\lambda =  e^{\ell(\gamma)}$. 
\end{theorem}

Note that we have $\deg \lambda = 2$ in Theorem \ref{thm:1} only if $K = \rationalnos$. See Theorem
\ref{thm:2} below for a result without the assumption $\gamma \in \Gamma^{(2)}$ for $n$ odd.
Theorem \ref{thm:1} has the following corollary with no restriction on dimension. 

\begin{corollary}  % 1 ( 1.2 )
\label{cor:1}
Let $\lambda$ be a Salem number. 
Then for each integer $n\geq 2$, there exist an arithmetic lattice $\Gamma \subseteq\Isom(H^n)$
of the simplest type defined over $\rationalnos(\lambda + \lambda^{-1})$ and a hyperbolic element  
$\gamma$ of $\Gamma$ such that $\lambda = e^{\ell(\gamma)}$. 
\end{corollary}

The set $\mathcal{S}_d$ of all Salem numbers of degree $d$ has a least element
(see Lemma 3 of \cite{G-H}). 
All non-cocompact arithmetic lattices in $\Isom(H^n)$ are arithmetic lattices of the simplest type 
defined over $\rationalnos$ \cite{L-M}. 
For each even integer $n\geq 2$, let 
$$b_n = \mathrm{min}\{\log \lambda : \lambda\ \hbox{is a Salem number with $\deg \lambda \leq n$}\}.$$
The next corollary follows from Theorem \ref{thm:1} and a sharp example for $n = 2$. 

\begin{corollary}  % 2 ( 1.3 )
\label{cor:2}
If $\Gamma \subseteq\Isom(H^n)$ is a non-cocompact arithmetic lattice, with $n$ even, 
and $C$ is a closed geodesic in $H^n/\Gamma$,  
then $\mathrm{length}(C) \geq b_n$,  and this lower bound is sharp for each even $n \geq 2$. 
\end{corollary}

Let $\lambda_{m,\ell}$ be the $\ell$th largest Salem number of degree $m$. 
The Salem numbers $\lambda_{m,1}$ for $m \leq 10$ are listed in \cite{L} and decrease as $m$ increases. 
 We conclude that $b_n = \log(\lambda_{n,1})$ for $n\leq 10$.  
The smallest known Salem number is Lehmer's number $\lambda_{10,1} = 1.1762808182\ldots$\ .
The values of $\lambda_{m,1}$ for $m \leq 54$ have been determined \cite{MRW}, and so 
$$b_{10} =  b_{12} = \cdots = b_{54} = 0.1623576120\ldots\ . $$
For fixed $m \le 24$, lists of the first few $\left\{\lambda_{m, 1}, \lambda_{m,2}, \dots \right\}$
can be found in \cite{Moss}.

\subsection{Lehmer's problem} % 1.2
\label{sec:intro-lehmer}
``Lehmer's problem'' refers to the question whether the Mahler measure of an irreducible noncyclotomic
polynomial in $\Z[x]$ can be arbitrarily closed to $1$. It is largely believed that the answer is
``no'', and that Lehmer's number $\lambda_{10,1}$ realizes the smallest possible value
(see \cite[\S\,2]{Smyth} and the survey \cite{Smyth2}). A weaker formulation is the following conjecture.
\begin{conjecture} % 1.4
        \label{conj:Lehmer}
        Lehmer's number $\lambda_{10,1}$ is the smallest Salem number.
\end{conjecture}

By an \emph{arithmetic hyperbolic orbifold} we mean a quotient $H^n/\Gamma$, with $\Gamma
\subseteq\Isom(H^n)$ an arithmetic lattice (in particular, such a quotient has finite volume).
Since non-cocompact arithmetic lattices $\Gamma \subseteq\Isom(H^n)$ 
are all of the simplest type defined over $\Q$, 
it follows directly from Theorem \ref{thm:1} that Conjecture \ref{conj:Lehmer} is equivalent to
the following.
\begin{conjecture} % 1.5
        \label{conj:geod}
       The minimal possible length of a closed geodesic in a noncompact, arithmetic, 
       even-dimensional, hyperbolic orbifold is $b_{10} = \log(\lambda_{10,1})$.
\end{conjecture}
Conjecture \ref{conj:geod} is a variation of the classical ``short geodesic conjecture'',
which predicts that there is a minimal length amongst closed geodesics on arithmetic
hyperbolic surfaces (see \cite[Conjecture~11]{G-H}). The formulation in Conjecture \ref{conj:geod} differs
from this classical version in the sense that it fixes the field of definition to be $\Q$, and allows
arbitrarily large dimensions.

Note that Theorem \ref{thm:1} also suggests -- together with Conjecture \ref{conj:Lehmer}
-- that $\frac{1}{2} \log(\lambda_{10,1})$ is a lower bound for the length of a closed
geodesic on any arithmetic hyperbolic $n$-orbifold of the simplest type, with $n$ odd. We will discuss again the
specific case of those orbifolds that are non-compact in Conjecture \ref{conj:geod-odd} below.

\subsection{Square-rootable Salem numbers}  % 1.3
\label{sec:intro-square-rootable}

Let $\lambda$ be a Salem number, let $K$ be a subfield of $\rationalnos(\lambda +\lambda^{-1})$, 
and let $p(x)$ be the minimal polynomial of $\lambda$ over $K$.  
We say that $\lambda$ is {\it square-rootable over} $K$ if there exist a totally positive element $\alpha$ of $K$ 
and a monic palindromic polynomial $q(x)$, whose even degree coefficients are in $K$ and whose odd degree coefficients 
are in $\sqrt{\alpha}K$, such that $q(x)q(-x) = p(x^2)$. 

\begin{theorem} % 2 ( 1.6 )
\label{thm:2}
Let $\Gamma \subseteq\Isom(H^n)$ be an arithmetic lattice, with $n$ odd,  
of the simplest type defined over a totally real number field $K$.    
Let $\gamma$ be a hyperbolic element of $\Gamma$, and let $\lambda = e^{2\ell(\gamma)}$.  
Then $\lambda$ is a Salem number which is square-rootable over $K$.

Conversely, 
if $\lambda$  is a Salem number and $K$ is a subfield of $\rationalnos(\lambda +\lambda^{-1})$, 
and $n \geq 3$ is an odd integer with $\deg_K(\lambda) \leq n+1$, and $\lambda$ is square-rootable over $K$, 
then there exist an arithmetic lattice $\Gamma \subseteq\Isom(H^n)$ of the simplest type defined 
over $K$ and a hyperbolic element $\gamma$ in $\Gamma$ such that $\lambda^{\frac{1}{2}} =  e^{\ell(\gamma)}$. 
\end{theorem}

Any Salem number $\lambda$ is square-rootable over $\rationalnos(\lambda+\lambda^{-1})$, 
and so Theorem \ref{thm:2} implies the next Corollary, which improves Corollary
\ref{cor:1} in odd dimensions.

\begin{corollary}  % 3 ( 1.7 )
\label{cor:3}
Let $\lambda$ be a Salem number. 
Then for each odd integer $n \geq 3$, 
there exist an arithmetic lattice $\Gamma \subseteq\Isom(H^n)$ of the simplest type 
defined over $\rationalnos(\lambda+\lambda^{-1})$ and a hyperbolic element $\gamma$ of $\Gamma$ 
such that $\lambda^{\frac{1}{2}} = e^{\ell(\gamma)}$. 
\end{corollary}

For each odd positive integer $n$, let 
\begin{eqnarray*}
c_n & = & \mathrm{min}\{\textstyle{\frac{1}{2}}\log \lambda : \lambda\  \hbox{is a Salem number with} \ \deg \lambda \leq n+1, \\
& & \hspace{.32in} \hbox{which is square-rootable over}\  \mathbb{Q}\}. 
\end{eqnarray*}
The next corollary follows from Theorem \ref{thm:2} and a sharp example for $n = 3$. 

\begin{corollary}  % 4 (1.8 )
\label{cor:4}
If $\Gamma \subseteq\Isom(H^n)$ is a non-cocompact lattice, with $n$ odd, 
and $C$ is a closed geodesic in $H^n/\Gamma$,  
then $\mathrm{length}(C) \geq c_n$,   
and this lower bound is sharp for each odd integer $n \geq 3$. 
\end{corollary}

Our lower bound 
$$c_3 = \textstyle{\frac{1}{2}}\log \lambda_{4,6} = 0.4312773138 \ldots\, $$ 
for the length of a closed geodesic in an arithmetic, noncompact, hyperbolic 3-orbifold $H^3/\Gamma$ agrees with the 
lower bound given in \cite{N-R}. 

We also have determined (see \S\,\ref{sec:9}) that 
$$c_5 = c_7 = \textstyle{\frac{1}{2}}\log \lambda_{6,4} = 0.2294546519 \ldots\ $$
and 
$$c_9 = c_{11} = \cdots = c_{19} = b_{10} = 0.1623576120 \ldots\ .$$
This suggests a possible extension of Conjecture \ref{conj:geod} to include the case of odd-dimensional 
non-compact orbifolds. We state this in the following conjecture, which at this point should be considered 
as more speculative than Conjecture \ref{conj:geod}.

\begin{conjecture} % 1.9
        \label{conj:geod-odd}
       The minimal possible length of a closed geodesic in any noncompact arithmetic
       hyperbolic orbifold is $b_{10} = \log(\lambda_{10,1})$.
\end{conjecture}

\subsection{Outline}
\label{sec:outline}
Our paper is organized as follows:  In \S\,\ref{sec:2}, we present background material for the paper. 
In \S\,\ref{sec:3}, we prove some preliminary algebraic lemmas. 
In \S\,\ref{sec:4}, we prove some linear algebraic group lemmas. 
In \S\,\ref{sec:5}, we prove the first half of Theorem \ref{thm:1}. 
In \S\,\ref{sec:7}, we prove the second half of Theorem \ref{thm:1}.  
In \S\,\ref{sec:8}, we prove Theorem \ref{thm:2}. 
In \S\,\ref{sec:9}, we determine the values of $c_n$ for odd $n \leq 19$. 
In \S\,\ref{sec:10}, we give an example with $K$ an intermediate field between $\rationalnos$ 
and $\rationalnos(\lambda+\lambda^{-1})$. 

\subsection*{Acknowledgements}
We thank David Boyd and Alan Reid for helpful correspondence,  
and Ted Chinburg, Olivier Mila, and the referees for comments that helped improve  
the exposition of the paper. 
We also thank AIM for providing a congenial work environment for the authors during a SQuarRE on hyperbolic geometry 
beyond dimension three.  The first author is supported by the SNSF, project no. PP00P2\_157583.

%%%%%%%%%%%%%%%%%%
\section{Background and notation} % 2
\label{sec:2}

\subsection{Salem numbers and polynomials}
%\label{sec:Salem-poly}

Let $\lambda$ be a Salem number, and let $s(x)$ be its \emph{Salem polynomial}, i.e., the
minimal polynomial of $\lambda$ over $\Q$.   We refer to \cite{G-H} for standard facts
about Salem polynomials. Let us recall here that $s(x)$ is over $\Z$ and the roots of $s(x)$ occur in pairs of
reciprocal numbers, namely $\left\{ \lambda, \lambda^{-1} \right\}$ and pairs of
complex conjugate numbers on the unit circle. 
In particular the degree  $m=\deg\lambda$ is an even positive integer.
Moreover, $s(x)$ is a palindromic polynomial, that is, $s(x) = x^m s(x^{-1})$.

%%%%%%%%

\subsection{The hyperboloid model for $H^n$}
%\label{sec:hyperb-model}
Let $f$ be a quadratic form in $n+1$ variables with  
a real symmetric coefficient matrix $A = (a_{ij})$. 
Then we have  
$f(x) = x^tAx.$
Let $R$ be a subring of $\C$.
We say that $f$ is {\it over} $R$ if $a_{ij} \in R$ for all $i, j$. 
The {\it orthogonal group} of $f$ over $R$ is defined to be
\begin{align*}
{\rm O}(f,R) &= \{T \in {\rm GL}(n+1,R) : f(Tx) = f(x) \ \hbox{for all} \ x \in
\realnos^{n+1}\} \\
&= \{T\in {\rm GL}(n+1,R): T^tAT =A\}.
\end{align*}
The {\it orthogonal group} of $f$ is ${\rm O}(f) = {\rm O}(f,\C)$.

Let $f_n$ be the {\it Lorentzian quadratic form} in $n+1$ variables given by
$$f_n(x) = x_1^2+ \cdots + x_n^2 - x_{n+1}^2.$$
Then ${\rm O}(f_n,\realnos) = {\rm O}(n,1)$. 
The hyperboloid model of {\it hyperbolic $n$-space} is 
$$H^n = \{x \in \realnos^{n+1} : f_n(x) = -1 \ \ \hbox{and} \ \ x_{n+1} > 0\}.$$
Let ${\rm O}^+(n,1)$ be the subgroup of ${\rm O}(n,1)$ consisting of all $T\in {\rm O}(n,1)$ that leave $H^n$ invariant. 
Then ${\rm O}^+(n,1)$ has index 2 in ${\rm O}(n,1)$. 
Restriction induces an isomorphism from ${\rm O}^+(n,1)$ to ${\rm Isom}(H^n)$, and we may
identify these two groups.

Suppose that the quadratic form $f$ has signature $(n,1)$. 
This means that there exists $M \in \mathrm{GL}(n+1,\realnos)$ such that $f(Mx) = f_n(x)$ for all $x \in \realnos^{n+1}$. 
Then $M$ maps the set $\{x \in \realnos^{n+1}: f_n(x) < 0\}$ onto the set $\{x\in \realnos^{n+1}: f(x) < 0\}$. 
Hence the set $\{x\in \realnos^{n+1}: f(x) < 0\}$ is a cone with two connected components. 
If $R$ is a subring of $\realnos$,  
let $\mathrm{O}'(f,R)$ be the subgroup of $\mathrm{O}(f,R)$ consisting of all $T\in \mathrm{O}(f,R)$ that 
leave both components of the cone $\{x \in \realnos^{n+1} : f(x) < 0 \}$ invariant. 
Then $\mathrm{O}'(f, R)$ has index 2 in $\mathrm{O}(f, R)$, and  
\begin{align}
        M\mathrm{O}^+(n,1)M^{-1} &= \mathrm{O}'(f,\realnos).
        \label{eq:isom-M}
\end{align}

\subsection{Arithmetic groups of the simplest type}
\label{sec:groups-simplest}

Let $K \subseteq\R$ be a totally real number field, and   
let $\mathfrak{o}_K$ be its ring integers.
Let $f$ be a quadratic form over $K$ in $n+1$ variables with coefficient matrix $A =(a_{ij})$. 
The quadratic form $f$ is said to be {\it admissible}  if $f$ has signature $(n,1)$, 
and for each nonidentity embedding $\sigma: K \to \realnos$ the quadratic form $f^\sigma$ 
over $\sigma(K)$, with coefficient matrix $A^\sigma = (\sigma(a_{ij}))$, is positive definite. 

A subgroup $\Gamma$ of $\mathrm{O}^+(n,1)$ is an {\it arithmetic subgroup of the simplest type 
defined over} $K$ if there exists an admissible quadratic form $f$ over $K$ in $n+1$ variables,  
and there exists $M \in \mathrm{GL}(n+1,\realnos)$ such that $f(Mx) = f_n(x)$ for all $x \in \realnos^{n+1}$,  
and the subgroups $M\Gamma M^{-1}$ and $\mathrm{O}'(f,\mathfrak{o}_K)$ 
of $\mathrm{O}'(f,\realnos)$ are commensurable, that is,  
$M\Gamma M^{-1} \cap \mathrm{O}'(f,\mathfrak{o}_K)$
has finite index in both $M\Gamma M^{-1}$ and $\mathrm{O}'(f,\mathfrak{o}_K)$. 
In this case $\Gamma$ is a lattice of $\O^+(n,1)$, i.e., $\Gamma$ is a discrete subgroup of finite covolume 
in $\O^+(n,1)$.
In this paper, such a subgroup $\Gamma$ will be called {\it classical} if moreover  
$M$ and $f$ can be taken so that $M\Gamma M^{-1} \subseteq \mathrm{O}'(f,K)$.

\vspace{.1in}
\noindent{\bf Notation Variance.} 
Equation \eqref{eq:isom-M} provides an isomorphism between $\O'(f, \R)$ and $\O^+(n,1)$. 
We will identify $\O'(f, \R)$ with $\Isom(H^n)$ and replace $M\Gamma M^{-1}$ by $\Gamma$ 
in order to simplify notation when the matrix $M$ plays no essential role.  

\subsection{Arithmetic quotients} % 2.4
Let $\Gamma \subseteq\Isom(H^n)$ be an arithmetic lattice of the simplest type defined over $K$ 
with respect to an admissible quadratic form $f$. 
The hyperbolic orbifold $H^n/\Gamma$ is compact unless $K = \rationalnos$ 
and there exists $x \neq 0$ in $\rationalnos^{n+1}$ such that $f(x) = 0$ (see \S 12 of \cite{B-H}). 
Suppose that $K = \rationalnos$. 
If $n = 2, 3$, then $H^n/\Gamma$ is compact for some $f$ and not compact for some $f$.   
If $n > 3$, then $H^n/\Gamma$ is not compact, 
since there exists $x \neq 0$ in $\rationalnos^{n+1}$ such that $f(x) = 0$ (see \cite{C} p 75).

%%%%%%%%%%%%%%%%%%%

\section{Preliminary Algebraic Lemmas} % 3
\label{sec:3}

The point of departure of our work in this paper is our first lemma, which was motivated 
by Takeuchi's lemma \cite{T}. In this section $K$ denotes a number field and $\o_K$ its
ring of integers. The symbol $\A$ stands for the full ring of integers in $\C$.

\begin{lemma} % 1 ( 3.1 )
\label{lem:1}
Let $A$ be an $n \times n$ matrix with coefficients in $K$ such that 
$A^m$ is over $\mathfrak{o}_K$ for some positive integer $m$. 
Then the characteristic polynomial $\mathrm{char}(A)$ of $A$ is over $\mathfrak{o}_K$. 
\end{lemma}
\begin{proof}
The roots of $\mathrm{char}(A^m)$ are the $m$th powers of the roots of $\mathrm{char}(A)$.
Since $\mathrm{char}(A^m)$ is over $\o_K$ and $\A$ is integrally closed,
it follows that the roots of $\mathrm{char}(A)$ are in $\A$. Thus its coefficients are in
$\A \cap K = \o_K$.
\end{proof}

The next lemma generalizes a well-known lemma ($K = \rationalnos$) for Salem polynomials. 

\begin{lemma} % 4 ( 3.2 )
\label{lem:4}
Let $K \subseteq\R$, and let $p(x)$ be an irreducible monic polynomial over $\mathfrak{o}_K$ of 
degree $m=2\ell$ whose real roots are $\lambda$ 
and $\lambda^{-1}$, with $\lambda > 1$, and whose complex roots have absolute value equal to 1. 
Then there exists a unique monic irreducible polynomial $q(x)$ over $\mathfrak{o}_K$ of degree $\ell$, 
called the trace polynomial of $p(x)$,  
such that $p(x) = x^\ell q(x+ x^{-1})$.  
\end{lemma}
\begin{proof}
        The number $\lambda + \lambda^{-1}$ is an algebraic integer, so its minimal
        polynomial over $K$  must have coefficients in $\o_K$. Let $q(x)$ be this minimal
        polynomial. Since the roots of $p(x)$ occur in $\ell$ pairs of reciprocal
        conjugate numbers, we see that there are exactly $\ell$ distinct embeddings 
        of $K(\lambda + \lambda^{-1})$ into $\C$ over $K$. 
        Thus $q(x)$ has degree $\ell$, and it follows
        that the monic polynomial $x^\ell q(x + x^{-1})$ must be $p(x)$.
\end{proof}

The next lemma and its proof was communicated to us by David Boyd. 

\begin{lemma}  % 5 ( 3.3 )
\label{lem:5}
The number field $K$ is totally real if and only if there is a Salem number $\lambda$ 
such that $K = \rationalnos(\lambda +\lambda^{-1})$. 
\end{lemma}
\begin{proof}
If $\lambda$ is a Salem number, then $K = \Q(\lambda + \lambda^{-1})$ is totally
real, since all the roots of the trace polynomial of the Salem polynomial of $\lambda$ are real (see Lemma \ref{lem:4}). 
Conversely, suppose $K$ is totally real.  
Then there exists a Pisot number $\alpha$ such that $K = \rationalnos(\alpha)$ by 
Hilfssatz 1 of \cite{P}. Then $2\alpha$ is an algebraic integer
all of whose remaining conjugates lie in the interval $(-2, 2)$. 
Let $\lambda$ be the largest solution of $x + x^{-1} = 2\alpha$. 
Then $\lambda$ is a Salem number, and  
$K = \rationalnos(2\alpha) = \rationalnos(\lambda+ \lambda^{-1})$. 
\end{proof}

Let $K \subseteq\R$ be a totally real number field, and let $\lambda$ be a Salem number such 
that $K = \Q(\lambda + \lambda^{-1})$.  It is then clear that the quadratic form
$$f(x) = x_1^2 + \cdots + x_n^2 - (\lambda+\lambda^{-1}-2)x_{n+1}^2$$
is admissible over $K$, and thus $\O'(f, \o_K)$ is a classical arithmetic subgroup of
$\Isom(H^n)$ of the simplest type defined over $K$.

%%%%%%%%%%%%%%%%%%%%%%%%%

\section{Linear Algebraic Group Lemmas} %4
\label{sec:4}

In this section, we prove some lemmas that require the theory of linear algebraic groups
\cite{B2}. 
Let $f$ be a quadratic form over a subfield $K$ of $\realnos$ of signature $(n,1)$. 
We are primarily interested in algebraic $K$-subgroups of $\mathrm{GL}(n+1,\complexnos)$ such as $\mathrm{O}(f)$
with the exception of the quotient algebraic $K$-group $\mathrm{PO}(f) = \mathrm{O}(f)/\{\pm I\}$ 
whose algebraic $K$-group structure is described in matrix terms in Lemma \ref{lem:7} below. 
For $G$ an algebraic $K$-group, we will denote its group of $K$-points by $G_K$.

\begin{lemma} % 4.1
        \label{lem:G-adjoint-K-points}
        If $G$ is an adjoint $K$-simple algebraic $K$-group, with $K \subseteq \R$, and $\Gamma \subseteq G_\R$ is 
        an arithmetic subgroup (i.e., commensurable with $G_{\o_K}$), then $\Gamma \subseteq G_K$.
\end{lemma}
\begin{proof}
        The result is well known for $G$ absolutely simple and adjoint, and is proved for instance in
        \cite[Prop.~1.2]{B-Pr}. One can use Weil restriction of scalars to deduce the
        general result from that particular case: if $G$ is adjoint $K$-simple, by the
        proof of \cite[Theorem~26.8]{KMRT} there exist a finite 
        field extension $L/K$ and an absolutely simple $L$-group $H$ such that
        $\mathrm{Res}_{L/K}(H) = G$. Since $G$ is adjoint, so must be $H$ and we conclude
        that $\Gamma \subseteq H_L = G_K$.
\end{proof}

Recall that by a \emph{classical} arithmetic subgroup of $\Isom(H^n)$ (of the simplest
type) we mean an arithmetic subgroup $\Gamma$ constructed in the $K$-points $\O'(f, K)$ for some
admissible quadratic form $f$ over $K$. For even dimensions this notion actually covers
all arithmetic hyperbolic lattices, as the following lemma shows.

\begin{lemma} % 6 ( 4.2 )
\label{lem:6}
If $\Gamma \subseteq\Isom(H^n)$ is an arithmetic lattice, with $n$ even, 
then $\Gamma$ is classical. 
\end{lemma}
\begin{proof}
For $n$ even, all arithmetic lattices of $\Isom(H^n)$ are of the simplest type.
We may assume that $\Gamma \subseteq \O'(f, \R)$ and $\Gamma$ is commensurable to $\O'(f, \o_k)$ 
for some admissible quadratic form $f$ in $n+1$ variables over $K \subseteq \R$. 
Let $\SO(f)$ be the special orthogonal group of $f$, which is an algebraic $K$-group. 
Define $\psi: \mathrm{O}(f) \to \mathrm{SO}(f)$ by $\psi(A) = (\det A)A$. Then $\psi$ 
a $K$-homomorphism, whose restriction to $\O'(f, \R)$ is an isomorphism onto $\SO(f, \R)$.
In particular $\psi$ maps $\O'(f, K)$ isomorphically onto $\SO(f, K)$. Since for $n$ even, 
$\SO(f)$ is absolutely simple and adjoint (see \cite[\S\,26.A]{KMRT}),
the arithmetic subgroup $\psi(\Gamma)$ must be contained in
$\SO(f)_K = \SO(f, K)$ by Lemma \ref{lem:G-adjoint-K-points}. It follows that $\Gamma \subseteq\O'(f, K)$, 
and so $\Gamma$ is classical. 
\end{proof}

The situation for $n$ odd is not as easy, since $\SO(f)$ is not adjoint in this case. We need to
introduce more notation to deal with it. 
Let $f$ be a nondegenerate quadratic form in $n+1$ variables over a subfield $K$ of $\realnos$ with $n$ odd, 
and let $A$ be the coefficient matrix of $f$.  
The {\it general orthogonal group} of $f$ (cf.\,\cite{KMRT} p 154) is the group
$$\mathrm{GO}(f) = \{B \in \mathrm{GL}(n+1,\complexnos) :  B^tAB = bA \ \hbox{for some}\ b\in \complexnos\}.$$
If $B \in \mathrm{GO}(f)$ and $B^tAB = bA$ with $b \in \complexnos$, then $bI = B^tABA^{-1}$, and so $b \neq 0$
and $b$ is uniquely determined by $B$. We write $\mu(B)$ for $b$.  
If $c\in \complexnos^\times$, then $cI \in \mathrm{GO}(f)$ with $\mu(cI) = c^2$. 
Hence, the map $\mu: \mathrm{GO}(f) \to \complexnos^\times$ 
is an epimorphism with kernel equal to $\mathrm{O}(f) = \mathrm{O}(f, \complexnos)$. 
If $B \in \mathrm{GO}(f)$ and $b = \mu(B)$,   
then $(\det B)^2 = b^{n+1}$, and so $\det B = \pm b^{(n+1)/2}$. 
Let $D = \{cI_{n+1}: c\in \complexnos^\times\}$.  Then $D$ is a normal subgroup of $\mathrm{GO}(f)$. 
The {\it projective general orthogonal group} of $f$ is the group $\mathrm{PGO}(f) = \mathrm{GO}(f)/D$.

Let $\mathrm{GO}(f, K) = \mathrm{GO}(f)\cap \mathrm{GL}(n+1,K)$. 
Suppose $B \in \mathrm{GO}(f, K)$.  Then $B^tAB = bA$ with $b = \mu(B)$. 
Hence $b \in K^\times$. 
Now $B$ represents an equivalence from $f$ to $bf$ over $K$, 
and so $f$ and $bf$ have the same signature $(p, q)$.  
If $p \neq q$, we must have $b > 0$. 
If $f$ is an admissible quadratic form over a totally real number field $K$ and $n \geq 3$, 
then we have that $\sigma(b) > 0$ for each embedding $\sigma: K \to \realnos$, 
that is, $b$ is a {\it totally positive} element of $K$. 

\begin{lemma} % 7 ( 4.3 )
\label{lem:7}
Let $f$ be a quadratic form in $n+1$ variables over a subfield $K$ of $\realnos$ of signature $(n,1)$ with $n$ odd 
and $n \geq 3$. 
Let $\mathrm{PO}(f)$ be the algebraic $K$-group $\mathrm{O}(f)/\{\pm I\}$, 
and let $\mathrm{PO}(f)_K$ be the group of $K$-points of $\mathrm{PO}(f)$. 
Then 
$$\mathrm{PO}(f)_K = \{\{\pm \textstyle{\frac{1}{\sqrt{b}}}B\}: B \in \mathrm{GO}(f,K)\ \hbox{and}\ b = \mu(B)\}.$$
\end{lemma}
\begin{proof}
Let $\pi: \mathrm{O}(f) \to \mathrm{PO}(f)$ and $\eta: \mathrm{GO}(f) \to \mathrm{PGO}(f)$ be the quotient maps. 
Then $\pi$ and $\eta$ are $K$-homomorphisms of algebraic $K$-groups by Theorem 6.8 of \cite{B2}. 
The inclusion map $\upsilon: \mathrm{O}(f) \to \mathrm{GO}(f)$ is a $K$-homomorphism.  
By the Universal Mapping Property (\cite{B2} p 94), the inclusion $\upsilon: \mathrm{O}(f) \to \mathrm{GO}(f)$ 
induces a $K$-homomorphism $\overline\upsilon: \mathrm{PO}(f) \to \mathrm{PGO}(f)$ 
such that $\overline\upsilon\pi = \eta\upsilon$. 
If $B \in \mathrm{O}(f)$,  
then $\overline\upsilon(\{\pm B\}) = DB$, and so $\overline\upsilon$ is a monomorphism.  
Now assume that $B \in \mathrm{GO}(f)$, and 
let $b = \mu(B)$.  Then $\det B = \pm b^{(n+1)/2}$. 
Hence $\det(\textstyle{\frac{1}{\sqrt{b}}}B) = \pm 1$. We have that $\overline\upsilon(\{ \pm \textstyle{\frac{1}{\sqrt{b}}}B\}) = DB$, 
and so $\overline\upsilon$ is onto, and therefore $\overline\upsilon$ is an isomorphism.
In particular the restriction to $K$-points $\overline\upsilon: \PO(f)_K \to \PGO(f)_K$ is
an isomorphism.

The short exact sequence of algebraic $K$-groups 
$$ 1 \to D\  {\hookrightarrow}\ \mathrm{GO}(f) \  {\buildrel \eta\over \longrightarrow}\ \mathrm{PGO}(f) \to 1$$
determines an exact sequence of Galois cohomology groups and homomorphisms
$$ 1 \to D_K\  {\longrightarrow}\ \mathrm{GO}(f)_K \  {\buildrel \eta\over \longrightarrow}\ 
\mathrm{PGO}(f)_K {\longrightarrow}\ H^1(K,D)$$
by the discussion in \S 1.3 of \cite{B-S} and Proposition 1.17 and Corollary 1.23 of \cite{B-S}. 
Since $H^1(K, D) = 0$ (by Proposition 1 on p 72 of \cite{Serre} and induction on $n$), we   
obtain $\eta(\mathrm{GO}(f)_K) = \mathrm{PGO}(f)_K$. 
We have that $\mathrm{GO}(f)_K = \mathrm{GO}(f,K)$, and so 
$$\mathrm{PGO}(f)_K = \{DB : B \in  \mathrm{GO}(f,K)\}.$$
Therefore 
$$ \mathrm{PO}(f)_K = \overline\upsilon^{-1}(\mathrm{PGO}(f)_K )= 
\{\{\pm \textstyle{\frac{1}{\sqrt{b}}}B\}: B \in \mathrm{GO}(f,K)\ \hbox{and}\ b = \mu(B)\}.$$

\vspace{-.25in}
\end{proof}

\begin{lemma} % 10 ( 4.4 )
\label{lem:9}
Let $f$ be an admissible quadratic from in $n+1$ variables over a totally real number field K, with $n$ odd and $n\geq 3$,  
let $\Gamma$ be a subgroup of ${\rm O}'(f,\R)$ that is commensurable to ${\rm O}'(f,\o_K)$,  
and let $\overline\Gamma$ be the image of $\Gamma$ in ${\rm PO}(f,\R)$. 
Then $\overline\Gamma \subseteq {\rm PO}(f)_K$.
\end{lemma}
\begin{proof}
Let $\pi: \O(f) \to \PO(f)$ be the natural projection, defined by $\pi(A) = \left\{\pm A \right\}$. 
It is a $K$-homomorphism that induces an isomorphism from $\O'(f, \R)$
to $\PO(f, \R)$ and an isomorphism from $\O'(f, K)$ to $\PO(f,K)$.

The group ${\rm PSO}(f)$ is adjoint, and $K$-simple, since ${\rm PSO}(f)$ is absolutely simple for
$n\neq 3$, and $\R$-simple for $n=3$ (since ${\rm PSO}(f, \R) \cong \mathrm{PGL}_2(\C)$ is a simple group).  
Let $\Gamma_0 = \Gamma\cap {\rm SO}'(f,\R)$.
Then $\pi(\Gamma_0) \subseteq {\rm PSO}(f)_K$ by Lemma \ref{lem:G-adjoint-K-points}. 
If $\Gamma = \Gamma_0$, we are done, so assume $\Gamma \neq \Gamma_0$. 
Let $\overline\Gamma = \pi(\Gamma)$ and $\overline\Gamma_0 = \pi(\Gamma_0)$. 

Let $\Lambda = {\rm SO}(f,\o_K)$ and $\overline\Lambda = {\rm PSO}(f,\o_K)$. 
If $H$ is a subgroup of $\O(f)$ (or $\PO(f)$), let $C(H)$ be the commensurator of $H$ in $\O(f)$ (or $\PO(f)$). 
There exists $R \in \O(f,K)$ with $\det R = -1$ by Theorem 3.20 of \cite{A}. 
Then $R\in C(\Lambda)$, since $\O(f,K) \subseteq C(\Lambda)$. 
Hence $\overline R = \pi(R)\in C(\overline{\Lambda})$ by Lemma 15.10 of \cite{B1}. 
We have that $C(\Gamma_0) = C(\Lambda)$, since $\Gamma_0$ and $\Lambda$ 
are commensurable. 
Hence $C(\overline\Gamma_0) = C(\overline\Lambda)$ by Lemma 15.10 of \cite{B1}. 

Let $B \in \Gamma$ with $\det B = -1$, and let $\overline B = \pi(B)$. 
Then $\overline B \in\overline\Gamma$, and so $\overline B \in C(\overline\Gamma_0)$. 
Hence $\overline R\overline B \in C(\overline\Gamma_0)$. 
If $H$ is a subgroup of ${\rm PSO}(f)$, let $C_0(H)$ be the commensurator of $H$ in ${\rm PSO}(f)$. 
As $\det(RB) = 1$, we have that $\overline R \overline B \in C_0(\overline\Gamma_0)$. 
In view of Lemma \ref{lem:7}, we have that ${\rm PSO}(f)_K \subseteq C_0(\overline\Lambda)$ by the same argument 
that ${\rm O}(f,K) \subseteq C(\Lambda)$.  
Hence $C_0(\overline\Lambda)$ is Zariski-dense in ${\rm PSO}(f)$, since 
${\rm PSO}(f)_K$ is Zariski-dense in ${\rm PSO}(f)$. Therefore
$$C_0(\overline\Gamma_0) = C_0(\overline\Lambda) = {\rm PSO}(f)_K$$
by the $K$-version of Lemma 15.11 of \cite{B1} (see {\it Remarques} on p 106 of \cite{B1}). 
Hence $\overline R \overline B \in {\rm PSO}(f)_K$.   
Therefore $\overline B\in \PO(f)_K$. 
Thus $\overline\Gamma \subseteq \PO(f)_K$. 
\end{proof}

\begin{lemma} % 10 ( 4.5 )
\label{lem:10}
Let $\Gamma \subseteq \Isom(H^n)$ be an arithmetic lattice of the simplest type defined over
a totally real number field $K$, with $n$ odd and $n\geq 3$, 
and let $\Gamma^{(2)}$ be the subgroup of $\Gamma$ generated by the squares of elements of $\Gamma$. 
Then $\Gamma^{(2)}$ is classical.
\end{lemma}
\begin{proof}
We may assume that $\Gamma\subseteq \O'(f,\R)$ and $\Gamma$ is commensurable to $\O'(f, \o_K)$ for some admissible quadratic form $f$ over $K$. 
If $\gamma \in \Gamma$, then $\gamma^2$ is over $K$ by Lemmas \ref{lem:7} and \ref{lem:9}. 
Hence $\Gamma^{(2)} \subseteq \O'(f, K)$, and so $\Gamma^{(2)}$ is classical.
\end{proof}

%%%%%%%%%%%%%%%%%%%%%%%%%%%%

\section{Translation lengths and Salem numbers}  % 5
\label{sec:5}

An isometry $\gamma$ of $H^n$ is {\it hyperbolic} if there exists a geodesic in $H^n$ 
along which $\gamma$ acts as a translation by a positive distance $\ell(\gamma)$. 
There are two types of hyperbolic isometries of $H^n$: {\it hyperbolic translations} 
and {\it loxodromic hyperbolic} isometries. 
An isometry $\gamma$ of $H^n$ is a hyperbolic translation if in the upper half-space model 
of hyperbolic $n$-space, $\gamma$ is conjugate to a magnification $\mu(x) = kx$ with $k > 1$.

Let $\gamma$ be an element of $\mathrm{O}^+(n,1)$.  
Define the {\it nonroot of unity degree}, $\mathrm{deg}_\infty(\gamma)$, of $\gamma$ to be 
the number of eigenvalues of $\gamma$ that are not roots of unity.  

\begin{lemma} % 11 ( 5.1 )
\label{lem:11}
Let $\gamma$ be a hyperbolic element of $\mathrm{O}^+(n,1)$. 
Then $\mathrm{deg}_\infty(\gamma)$ is even and $\mathrm{deg}_\infty(\gamma) \geq 2$  
with $\mathrm{deg}_\infty(\gamma) = 2$ if and only if there is a positive integer $m$ 
such that $\gamma^m$ is a hyperbolic translation. 
\end{lemma}
\begin{proof}
This follows easily from Proposition 1 of \cite{G}.
\end{proof}

Lemmas \ref{lem:6} and \ref{lem:10} and the next theorem imply the first half of Theorem
\ref{thm:1}. 

\begin{theorem} % 4 ( 5.2 )
\label{thm:4}
Let $\Gamma \subseteq \Isom(H^n)$ be a classical arithmetic lattice 
of the simplest type defined over a totally real number field $K$. 
Let $\gamma$ be a hyperbolic element of $\Gamma$, let $\ell(\gamma)$ be the translation length of $\gamma$, 
and let $\lambda = e^{\ell(\gamma)}$.  Then 
\begin{enumerate}
\item $\lambda$ is a Salem number, and 
$$
\mathrm{deg}_K(\lambda) = \mathrm{deg}_\infty(\gamma) \leq n+1;$$
\item $K$ is a subfield of $\mathbb{Q}(\lambda + \lambda^{-1})$,   
and  
$$[\mathbb{Q}(\lambda+\lambda^{-1}): K] = \mathrm{deg}_K(\gamma)/2.$$
\end{enumerate}
\end{theorem}
\begin{proof}
(1) We may assume that $\Gamma \subseteq \O'(f, K)$ and $\Gamma$  is commensurable to ${\rm O}'(f,\o_K)$ for some admissible 
quadratic form $f$ in $n+1$ variables over $K$. 
Since $\Gamma \cap \O'(f, \o_K)$ has finite index in $\Gamma$, 
there exists a positive integer $m$ such that $\gamma^m \in \mathrm{O}'(f,\mathfrak{o}_K)$. 
Let $p(x)$ be the characteristic polynomial of $\gamma$. 
By Lemma \ref{lem:1} we have that $p(x)$ is over $\o_K$.
The real roots of $p(x)$ are $\lambda$ and $\lambda^{-1}$, as simple roots, and possibly $\pm 1$, as simple or multiple roots, and 
the complex roots occur in complex conjugate pairs of the form $e^{\pm i\theta}$ for some real number $\theta$ 
with $0 < \theta < \pi$ by Proposition 1 of \cite{G}.

Let $p(x) = p_1(x)\cdots p_k(x)$ be a factorization of $p(x)$ into 
monic irreducible polynomials over $\mathfrak{o}_K$, where we assume that $\lambda^{-1}$
is a root of $p_1(x)$.
We claim that $\lambda$ is a root of $p_1(x)$. 
On the contrary, assume that $\lambda$ is not a root of $p_1(x)$. 
The complex roots of $p_1(x)$ occur in inverse pairs. 
The constant term of $p_1(x)$ is the product of the negatives of the roots of $p_1(x)$. 
Hence the constant term of $p_1(x)$ is $- \lambda^{-1}$, and so $\lambda^{-1} \in \o_K$. 
If $K = \Q$ this contradicts $0<\lambda^{-1}<1$. Assume then that $K \neq \Q$.
As $K$ is totally real, there exists a nonidentity embedding $\sigma: K \to \R$. 
Then $\sigma(\lambda^{-1}) \neq \pm 1$ is a real root of $p^\sigma(x)$. But the latter is the
characteristic polynomial of $\gamma^\sigma \in \O(f^\sigma, \R)$. Since this group is
compact, we obtain a contradiction. Thus $p_1(\lambda) = 0$. Since $p_1(x)$ is irreducible
over $\o_K$, it is the minimal polynomial of $\lambda$ over $K$. In particular,
$\deg_K(\lambda) = \deg p_1(x)$. The roots of $p_1(x)$ are conjugate to $\lambda$, and
thus cannot be roots of unity. 

Let $K^* \subseteq \R$ be the normal closure of $K/\Q$. 
Then there exist exactly $d=[K:\Q]$ embeddings $\sigma_1,\ldots, \sigma_d$ of $K$ into $K^*$.  
Consider the monic polynomial 
$p_1^*(x) = p_1^{\sigma_1}(x) \cdots p_1^{\sigma_d}(x).$
The Galois group $\Gal(K^*/\Q)$ acts on the 
embeddings $\left\{ \sigma_1, \dots, \sigma_d \right\}$ by composition on the left, and so
$p_1^*(x)$ is fixed under the action of $\Gal(K^*/\Q)$.  
As $K^*/\Q$ is Galois, we conclude that $p_1^*(x) \in \Z[x]$.

Assume that $\sigma_1$ is the identity embedding. 
Using the fact that $\O(f^{\sigma_i}, \R)$ is compact for $i>1$, 
we see that all roots of $p_1^*(x)$ besides $\lambda$ and $\lambda^{-1}$ are on the unit
circle. Therefore it suffices to show that $p_1^*(x)$ is irreducible over $\Z$ to conclude 
that $\lambda$ is a Salem number. 
To show this, let $g(x) \in \Z[x]$ be the
minimal polynomial of $\lambda$ over $\Q$. Then $g(x)$ divides $p_1^*(x)$ in $\Z[x]$,
so we can write $p_1^\ast(x) = g(x) h(x)$ with $h(x)$ a monic polynomial over $\integers$. 
Let $r$ be a root of $h(x)$. 
Then $r$ is a root of $p_1^{\sigma_j}(x)$ for some $j$. 
Now $p_1^{\sigma_j}(x)$ is the minimal polynomial of $r$ over $K^{\sigma_j}=\sigma_j(K)$. 
Therefore $p_1^{\sigma_j}(x)$ divides $h(x)$ in $K^{\sigma_j}[x]$. 
As $\sigma_j^{-1}$ fixes $h(x)$, we deduce that $p_1(x)$ divides $h(x)$ in $K[x]$. 
Hence $\lambda$ is a root of $h(x)$, 
which is a contradiction, since $\lambda$ is a simple root of $p_1^\ast(x)$. 
Therefore $g(x) = p_1^\ast(x)$. 
Thus $p_1^\ast(x)$ is irreducible over $\integers$.  

For $j > 1$, we define $p_j^*(x) = \prod_{i=1}^d p_j^{\sigma_i}(x) \in \Z[x]$ (same
argument as for $j=1$ above). Each root of $p_j(x)$ lies on the unit circle, and using the
compactness of $\O(f^{\sigma_i}, \R)$ we deduce that the same is true for all the roots of
$p_j^*(x)$. It follows that each root of $p_j^*(x)$ is a root of unity by Kronecker's Theorem \cite{K}. 
Thus the roots of $p_1(x)$ are precisely the roots of $p(x)$ that are not roots of unity. 
Therefore 
$$\mathrm{deg}_K(\lambda) = \mathrm{deg}(p_1(x)) = \mathrm{deg}_\infty(\gamma) \leq n+1.$$

(2)
Let $\deg p_1(x) = 2\ell$.
By Lemma \ref{lem:4}, there exists an irreducible monic polynomial $q(x)$ of degree $\ell$ over $\mathfrak{o}_K$ such that
$p_1(x) = x^\ell q(x+x^{-1})$.
Hence $q(x)$ is the minimal polynomial of $\lambda + \lambda^{-1}$ over $K$. 
Therefore 
$$[K(\lambda+\lambda^{-1}):K] = \mathrm{deg}(q(x)) =  \mathrm{deg}(p_1(x))/2.$$
Likewise we have 
\begin{eqnarray*}
[\rationalnos(\lambda+\lambda^{-1}): \rationalnos] & = & \mathrm{deg}(p_1^\ast(x))/2 \\
									   & = & \mathrm{deg}(p_1(x))[K:\rationalnos]/2 \\
									   & = & [K(\lambda + \lambda^{-1}):K][K:\rationalnos] \ \ 
									    = \ \ [K(\lambda + \lambda^{-1}):\rationalnos].
\end{eqnarray*}
As  $\rationalnos(\lambda+\lambda^{-1})$ is a subfield of $K(\lambda + \lambda^{-1})$, 
we deduce that $\rationalnos(\lambda+\lambda^{-1}) = K(\lambda + \lambda^{-1})$. 
Therefore $K$ is a subfield of $\rationalnos(\lambda+\lambda^{-1})$. 
Moreover,  we have that 
$$[\rationalnos(\lambda+\lambda^{-1}): K] = \mathrm{deg}(p_1(x))/2 = \mathrm{deg}_K(\gamma)/2.$$ 

\vspace{-.22in}
\end{proof}

\section{Salem numbers and translation lengths}  % 6
\label{sec:7}

In this section, we prove the second half of Theorem \ref{thm:1} and provide a sharp example for Corollary \ref{cor:2}
in dimension 2. We begin with a couple of lemmas. 

\begin{lemma}  % 14 ( 6.1 )
\label{lem:14}
If $\lambda$ is a Salem number and $K$ is a subfield of $\rationalnos(\lambda+\lambda^{-1})$, then
$$\deg(\lambda) = \deg_K(\lambda) [K: \rationalnos].$$
\end{lemma}
\begin{proof}
As $\rationalnos \subseteq K \subseteq \rationalnos(\lambda+\lambda^{-1})$, we have that 
$$\rationalnos(\lambda) \subseteq K(\lambda) \subseteq \rationalnos(\lambda+\lambda^{-1})(\lambda) = \rationalnos(\lambda),$$
and so $\rationalnos(\lambda) = K(\lambda)$.  Therefore we have that 
$$\deg(\lambda)  =  [\rationalnos(\lambda):\rationalnos]  =  [K(\lambda): K][K:\rationalnos]  =  \deg_K(\lambda) [K: \rationalnos]. $$

\vspace{-.22in}
\end{proof}

\begin{lemma} % 14.5 ( 6.2 )
\label{lem:14.5}
Let $\lambda$ be a Salem number, and let $p(x)$ be the minimal polynomial of $\lambda$ 
over a totally real number field $K \subseteq \R$. 
If $\deg \lambda = 2$, assume that $K =\Q$.
Then $p(x)$ is over $\o_K$, the real roots of $p(x)$ are $\lambda$ and $\lambda^{-1}$, 
the complex roots have absolute value equal to 1, 
the degree of $p(x)$ is even, and $p(x)$ is palindromic. 
\end{lemma}
\begin{proof}
This is clear if $\deg \lambda = 2$, so assume $\deg \lambda > 2$. 
Then $\lambda \not\in K$, since every subfield of $K$ is totally real and $\Q(\lambda)$ is not totally real. 
As $p(x)$ divides the Salem polynomial of $\lambda$, 
the roots of $p(x)$ are in $\A$, and so $p(x)$ is over $\A\cap K = \o_K$. 
Moreover the complex roots of $p(x)$ occur in inverse pairs on the unit circle, 
and the real roots are $\lambda$ and possibly $\lambda^{-1}$.  
In fact $\lambda^{-1}$ is a root of $p(x)$,  
since otherwise the constant term of $p(x)$ would be $-\lambda$ 
which is not in $K$. 
Hence the real roots of $p(x)$ are $\lambda$ and $\lambda^{-1}$ and $m=\deg p(x)$ is even. 
The constant term of $p(x)$ is $1$, and so $x^mp(x^{-1})$ is monic. 
Hence $p(x) = x^mp(x^{-1})$, since $x^mp(x^{-1})$ is the minimal polynomial of $\lambda^{-1}$ over $K$. 
Therefore $p(x)$ is palindromic. 
\end{proof}

\begin{theorem} % 6  ( 6.3 )
\label{thm:6}
Let $\lambda$ be a Salem number, let $K$ be a subfield of $\rationalnos(\lambda+\lambda^{-1})$, 
and let $n \geq 2$ be an integer with $\mathrm{deg}_K(\lambda) \leq n+1$. 
Then there exist a classical arithmetic lattice $\Gamma \subseteq \Isom(H^n)$ of the simplest type over $K$ 
and an orientation preserving hyperbolic element $\gamma \in \Gamma$ such that $\lambda = e^{\ell(\gamma)}$. 
\end{theorem}
\begin{proof}  Let $p(x)$ be the minimal polynomial of $\lambda$ over $K$. 
Then $\rationalnos(\lambda+\lambda^{-1})$ is totally real by Lemma \ref{lem:5}. 
Hence $K$ is totally real. 
Then $p(x)$ is over $\o_K$, the real roots of $p(x)$ are $\lambda$ and $\lambda^{-1}$, 
the complex roots have absolute value equal to $1$, 
and the degree of $p(x)$ is even by Lemma \ref{lem:14.5}

Let $\deg p(x) = m = 2 \ell$.
Let $r_1, \ldots, r_m$ be the roots of $p(x)$ with $r_{2j-1} = e^{-i\theta_j}$ and $r_{2j} = e^{i\theta_j}$, with $0 < \theta_j < \pi$,  
for $j = 1, \ldots, \ell-1$, and $r_{m-1} = \lambda^{-1}$ and $r_m = \lambda$. 
Let $\eta = \log \lambda$,
and define $M$ to be the block diagonal $m \times m$ matrix with blocks 
$$\left(\begin{array}{cr}  \cos \theta_j &-\sin \theta_j \\
                                     \sin \theta_j & \cos \theta_j
                                     \end{array}\right) 
\ \ \hbox{for}\ \ 1 \leq j < \ell,\ \ \hbox{and} \ \ 
\left(\begin{array}{cc}  \cosh \eta & \sinh \eta \\
                                     \sinh \eta & \cosh \eta
                                     \end{array}\right).$$
Then $M$ is a hyperbolic element of $\mathrm{O}^+(m-1,1)$ with characteristic polynomial $p(x)$, 
since the eigenvalues of $M$ are $e^{\pm i\theta_1},\ldots, e^{\pm i\theta_{\ell-1}}$, and $e^{\pm\eta}$. 
Moreover $\det M = 1$, and so $M$ is an orientation preserving isometry of $H^{m-1}$.

Define a vector $v$ in $\realnos^m$ by 
$$v = (1,0, 1, 0, \ldots, 1, 0).$$
Let $w_j = M^{j-1}v$ for $j = 1,\ldots, m$.
Then $w_j$ is the vector
$$\big(\cos(j-1)\theta_1,\sin(j-1)\theta_1, \ldots, \cos(j-1)\theta_{\ell-1},\sin(j-1)\theta_{\ell-1},\cosh(j-1)\eta,\sinh(j-1)\eta\big).$$
Let $W$ be the $m \times m$ matrix whose $j$th column vector is $w_j$. 
We claim that $W$ is invertible. 

Let $B$ be the block diagonal $m\times m$ matrix with the first $\ell-1$ blocks 
$$\left(\begin{array}{cr}  1 & -i\\
                                         1& i 
                                     \end{array}\right)\ \ \hbox{and last block} \ \ 
                                     \left(\begin{array}{cr}  1& -1\\
                                      				     1& 1
                                     \end{array}\right).$$
Observe that 
$$\left(\begin{array}{cr}  1 & -i\\
                                         1& i 
                                     \end{array}\right) 
 \left(\begin{array}{cc}  \cos(2k-1)\theta_j & \cos 2k\theta_j\\
                                    \sin(2k-1)\theta_j & \sin 2k\theta_j
                                     \end{array}\right)
=\left(\begin{array}{cc}  e^{-(2k-1)i\theta_j}&e^{-2k i\theta_j}\\
                                     e^{(2k-1)i\theta_j}& e^{2k i\theta_j}
                                   \end{array}\right), 
$$ 
$$\left(\begin{array}{cr}  1 & -1\\
                                         1& 1
                                     \end{array}\right) 
 \left(\begin{array}{cc}  \cosh(2k-1)\eta & \cosh 2k\eta\\
                                    \sinh(2k-1)\eta & \sinh 2k\eta
                                     \end{array}\right)
=\left(\begin{array}{cc}  e^{-(2k-1)\eta}&e^{-2k \eta}\\
                                     e^{(2k-1)\eta}& e^{2k \eta}
                                   \end{array}\right). $$                                  
Therefore we have that 
$$BW = \left(\begin{array}{ccccc}  1 & r_1 & r_1^2 & \cdots & r_1^{m-1} \\
						   1 & r_2 & r_2^2 & \cdots & r_2^{m-1} \\
						   \vdots & \vdots &\vdots &  & \vdots \\
						   1 & r_m & r_m^2 & \cdots & r_m^{m-1} 
						   \end{array}\right). $$  
Hence $V = BW$ is a Vandermonde $m \times m$ matrix. 
Therefore we have 
$$\det(V) = \prod_{1\leq j < k \leq m}(r_k-r_j),$$
and so $V$ and $W$ are invertible, since the roots $r_1,\ldots, r_m$ of $p(x)$ are distinct. 

Define an $m\times m$ matrix $C$ by the formula $C = W^{-1}MW$. 
Let $e_1, \ldots, e_m$ be the standard basis vectors of $\realnos^m$. 
Then for $j < m$, we have that 
$$Ce_j = W^{-1}MWe_j = W^{-1}Mw_j = W^{-1}w_{j +1}=e_{j+1}.$$
Therefore $C$ is of the form
$$C = \left(\begin{array}{ccccc}  0 & 0 &  \cdots & 0 & c_1 \\
						 1 & 0 &  \cdots & 0 & c_2 \\
						 0 & 1 &  \cdots & 0 & c_3 \\
						   \vdots & \vdots & & \vdots & \vdots \\
						 0 & 0 &  \cdots & 1 & c_m
						   \end{array}\right). $$  
The matrix $C$ has the same characteristic polynomial as $M$. 
Hence $C$ must be the companion matrix of $p(x)$, and so if 
$$p(x) = a_0 + a_1x +\cdots + a_{m-1}x^{m-1} +x^m,$$
then $c_j = -a_{j-1}$ for $j = 1,\ldots, m$. 
Therefore $C$ is over $\mathfrak{o}_K$. 

Define an $m\times m$ diagonal matrix $J$ by 
$$J = \mathrm{diag}(1,\ldots, 1, -1).$$
Then $J$ is the coefficient matrix of the Lorentzian quadratic form $f_{m-1}(x)$. 
Define a symmetric $m\times m$ matrix $A$ by the formula $A = W^tJW$. 
Then $A$ is the coefficient matrix of a quadratic form $f$ over $\realnos$ in $m$ variables.  
If $x \in \realnos^m$, then 
$$f(x) = x^tAx= x^tW^tJWx = (Wx)^tJWx = f_{m-1}(Wx),$$
and so $f$  has signature $(m-1,1)$ and 
$$\mathrm{O}'(f,\realnos) = W^{-1}\mathrm{O}^+(m-1,1)W.$$
Now $M \in \mathrm{O}^+(m-1,1)$ and $C= W^{-1}MW$. 
Hence $C\in \mathrm{O}'(f,\mathfrak{o}_K)$. 

We claim that $f$ is over $K$.  If $x, y \in \realnos^m$, 
define the {\it Lorentzian inner product} of $x$ and $y$ to be $x\circ y = x^tJy$. 
Let $A = (a_{jk})$. 
Then we have that $a_{jk} = w_j\circ w_k$. 
The matrix $M$ preserves the Lorentzian inner product. 
Hence if $j, k < m$, we have that 
$$a_{j+1, k+1} = w_{j+1} \circ w_{k+1} = Mw_j \circ Mw_k = w_j \circ w_k = a_{jk}.$$
Therefore $A$ is a Toeplitz matrix (diagonal-constant matrix). 
As $A$ is symmetric, to determine $A$ it suffices to determine the first column of $A$, that is, 
to determine $a_{j1}$ for $j =1, \ldots, m$. 
We have that 
$$a_{j1} = w_j\circ v = \cos(j-1)\theta_1+\cdots +\cos(j-1)\theta_{\ell-1} + \cosh(j-1)\eta.$$
For $j = 1$, we see that all the elements on the main diagonal of $A$ are equal to $\ell$. 
Now we have 
\begin{eqnarray*}
r_{2k-1}^{j-1} & = & \cos(j-1)\theta_k-i\sin(j-1)\theta_k \\
r_{2k}^{j-1}    & = & \cos(j-1)\theta_k+i\sin(j-1)\theta_k, 
\end{eqnarray*}
and so 
$$\cos(j-1)\theta_k  = (r_{2k-1}^{j-1} +r_{2k}^{j-1} )/2.$$
Likewise we have
\begin{eqnarray*}
r_{m-1}^{j-1} & = & \cosh(j-1)\eta- \sinh(j-1)\eta \\
r_{m}^{j-1}    & = & \cosh(j-1)\eta+ \sinh(j-1)\eta,
\end{eqnarray*}
and so 
$$\cosh(j-1)\eta  = (r_{m-1}^{j-1} +r_{m}^{j-1} )/2.$$
Therefore we have that 
$$a_{j1} = (r_1^{j-1} + \cdots + r_m^{j-1})/2.$$
Now $r_1^{j-1} + \cdots + r_m^{j-1}$ is equal to the symmetric polynomial $x_1^{j-1}+ \cdots + x_m^{j-1}$ 
evaluated at the roots $r_1, \ldots, r_m$ of $p(x)$. 
Let $s_k(x_1,\ldots, x_m)$ be the $k$th elementary symmetric polynomial in $m$ variables, 
and let 
$$t_k = s_k(r_1,\ldots, r_m).$$
Then we have that 
$$p(x) = x^m -t_1x^{m-1} + \cdots + (-1)^mt_m,$$
and so $t_k \in \mathfrak{o}_K$ for each $k = 1, \ldots, m$. 
By Newton's identities there is a polynomial $g_j(x_1,\ldots, x_m)$ over $\integers$ 
such that 
$$x_1^{j}+ \cdots + x_m^{j} = g_j(s_1(x_1,\ldots, x_m), \ldots, s_m(x_1,\ldots, x_m)).$$
Hence we have 
$$r_1^{j-1} + \cdots + r_m^{j-1} = g_{j-1}(t_1,\ldots, t_m).$$
Therefore $r_1^{j-1} + \cdots + r_m^{j-1} \in \mathfrak{o}_K$. 
Hence $2A$ is over $\mathfrak{o}_K$, and so $A$ is over $K$. 
Therefore the quadratic form $f$ is over $K$.  
In fact $f$ has collected coefficients in $\mathfrak{o}_K$, since $a_{jj} = \ell$ 
and if $j\neq k$, then $a_{jk} + a_{kj} \in \mathfrak{o}_K$. 

We next show that $f$ is admissible.  This is clear if $K = \rationalnos$, 
and so assume $K \neq \rationalnos$. 
Let $d = [K:\rationalnos]$, and let $\sigma_1, \ldots,\sigma_d$ be the embeddings of $K$ into $\realnos$ with $\sigma_1$ the inclusion of $K$ into $\realnos$. 
Define the monic polynomial $p^\ast(x) = p^{\sigma_1}(x)\cdots p^{\sigma_d}(x)$.  
Then $p(x) \in \Q[x]$ and $s(x)$ divides $p^*(x)$ in $\Q[x]$, since $\lambda$ is a root of $p^*(x)$. 
By Lemma \ref{lem:14}, we have 
$$\deg(p^\ast(x)) = \deg(p(x))d = \deg_K(\lambda)[K:\rationalnos] = \deg(\lambda) = \deg(s(x)),$$
and so $s(x) = p^*(x)$.

Assume that $j > 1$. 
Then  the roots of $p^{\sigma_j}(x)$ are simple complex roots that occur in complex conjugate pairs of the form $e^{\pm i\theta}$ 
for some real number $\theta$. 
Define an $m\times m$ block diagonal matrix $M_j$ in terms of the roots of $p^{\sigma_j}(x)$ in the same way that we defined $M$ . 
Then $M_j$ is a rotation matrix. 
Define an $m\times m$ matrix $W_j$ in terms of $M_j$ and $v$ in the same way that we defined $W$.  
Then $W_j$ is invertible by the same Vandermonde determinant argument. 
Define an $m \times m$ symmetric matrix $A_j$ by the formula $A_j = W_j^tW_j$. 
Then the quadratic form $f_j$, whose coefficient matrix is $A_j$, is positive definite. 
The entries of $A_j$ are expressed in terms of the coefficients of $p^{\sigma_j}(x)$ 
in the same way that the entries of $A$ are expressed in terms of the coefficients of $p(x)$.
Hence $A_j = A^{\sigma_j}$ and so $f_j = f^{\sigma_j}$.
Therefore $f^{\sigma_j}$ is positive definite. 
Thus $f$ is admissible. 

Let $\Gamma = W\mathrm{O}'(f,\mathfrak{o}_K)W^{-1}$. 
Then $\Gamma$ is a classical arithmetic group of isometries of $H^{m-1}$ 
of the simplest type over $K$. 
We have that $\gamma = M = WCW^{-1} $ is an orientation preserving hyperbolic element of $\Gamma$  
such that $\lambda = e^{\ell(\gamma)}$. 
If $m = n+1$ we are done, 
otherwise we boost $f$ to the quadratic form
$$x_1^2+\cdots + x^2_{n-m+1} + f(x_{n-m+2}, \ldots, x_{n+1}).$$

\vspace{-.23in}
\end{proof}

The next corollary is an enhanced version of Corollary \ref{cor:1}. 

\begin{corollary}  % 5 ( 6.4 )
\label{cor:5}
Let $\lambda$ be a Salem number. 
Then for each integer $n \geq 2$, there exist a classical arithmetic lattice $\Gamma \subseteq \Isom(H^n)$
of the simplest type defined over $\rationalnos(\lambda + \lambda^{-1})$ and a hyperbolic translation 
$\gamma$ in $\Gamma$ such that $\lambda = e^{\ell(\gamma)}$. 
\end{corollary}
\begin{proof} 
Let $K=\rationalnos(\lambda + \lambda^{-1})$.  Then $\deg_K(\lambda) = 2$ by Theorem \ref{thm:4}(2).  
Hence $\deg_K(\lambda) \leq n+1$ for $n \geq 1$. 
Therefore $\lambda = e^{\ell(\gamma)}$ for some $\gamma \in \Gamma$ given by Theorem \ref{thm:6}. 
Moreover $\gamma$ is a hyperbolic translation by the proof of Theorem \ref{thm:6}.  
\end{proof}

The next corollary follows from Lemma \ref{lem:6} and Theorems \ref{thm:4} and \ref{thm:6}. 

\begin{corollary} % 6 ( 6.5 )
\label{cor:6}
Let $\Gamma \subseteq {\rm Isom}(H^n)$ be an arithmetic lattice of the simplest type defined over $\rationalnos$, with $n$ even, 
and let $C$ be a closed geodesic in $H^n/\Gamma$.   
Then $\mathrm{length}(C) \geq b_n$,  and this lower bound is sharp for each even integer $n \geq 2$. 
\end{corollary}

Corollary \ref{cor:2} follows from Corollary \ref{cor:6} once we have a sharp non-cocompact example for $n =2$, 
since all arithmetic lattices of $\Isom(H^n)$ of the simplest type over $\rationalnos$ 
are not cocompact when $n > 3$.  
For $\deg\lambda = 2$, we have $\lambda + \lambda^{-1} \in \Z_{>2}$.
Therefore the smallest Salem number $\lambda_{2,1}$ of degree 2 occurs when $\lambda+\lambda^{-1} =3$, 
and so
$\lambda_{2,1} = \big(3+\sqrt{5}\big)/2$,  
and we have that 
$$b_2 = \log(\lambda_{2,1}) = .9624236501\ldots\ .$$
When $n = 2$ and $\lambda = \lambda_{2,1}$, the proof of Theorem \ref{thm:6} yields the quadratic form 
$$f(x) = x_1^2 + x_2^2 + 3x_2x_3 + x_3^2.$$
Now $f(1,-1,2) = 0$, and so
a corresponding arithmetic group $\Gamma$ of isometries of $H^2$ is not cocompact. 
Thus we have a sharp example for Corollary \ref{cor:2} when $n = 2$.

%%%%%%%%%%%%%%%%%%%%%%%%%%%

\section{Square-rootable Salem numbers}  % 8 ( 7 )
\label{sec:8}

In this section, we prove Theorem \ref{thm:2}.  
The first half of Theorem \ref{thm:2} is Theorem \ref{thm:7}, and the second half is
Theorem \ref{thm:8} below. We start with a few lemmas. The proof of the following, which is easy,
can be found in \cite[Lemma 2]{Smyth}. 
 
\begin{lemma}  % 12  ( 7.1 )
\label{lem:12}
If $\lambda$ is a Salem number of degree $d$, then $\lambda^k$ is a Salem number of degree $d$ for 
each positive integer $k$. 
\end{lemma}

It is clear that $\deg \lambda^{\frac{1}{2}}$ is either $\deg \lambda$ or $2 \deg
\lambda$. These two cases are distinguished by the following result.

\begin{lemma} % 15 ( 7.2 )
\label{lem:15}
Let $\lambda$ be a Salem number. Then  $\deg \lambda^{\frac{1}{2}} = \deg \lambda$ if and only if either
$\deg \lambda^{\frac{1}{2}} = 2$ or $\lambda^{\frac{1}{2}}$ is a Salem number. 
\end{lemma}
\begin{proof}
Suppose $\deg \lambda^{\frac{1}{2}} = \deg \lambda$.  Let $s(x) \in \Z[x]$ be the Salem
polynomial of $\lambda$. 
The roots of $s(x^2)$ are of the form $\pm r^{\frac{1}{2}}$ where $r$ is a root of $s(x)$. 
Let $p(x) \in \Z[x]$ be the minimal polynomial of $\lambda^{\frac{1}{2}}$ over $\Q$.  
We can write $s(x^2) = p(x) q(x)$ for some $q(x) \in \Z[x]$. 
None of the roots of $q(x)$ are roots of unity.  Hence $q(x)$ must have a real root
by Kronecker's theorem \cite{K}. 
As $\deg q(x) = \deg p(x) = \deg s(x)$, we have that $\deg q(x)$ is even. 
Therefore $q(x)$ has two real roots and $p(x)$ has two real roots.
If $\deg s(x) = 2$, then $\deg p(x) =  2$, and so $\deg  \lambda^{\frac{1}{2}} = 2$.
Suppose $\deg s(x) \geq 4$. 
Then $p(x)$ has a pair of reciprocal complex roots, whence 
all the roots of $p(x)$ occur in reciprocal pairs.  
Therefore $\lambda^\frac{1}{2}$ is a Salem number with Salem polynomial $p(x)$. 

Conversely, if $\deg \lambda^{\frac{1}{2}} = 2$, then $\deg \lambda = 2 = \deg \lambda^{\frac{1}{2}}$. 
If $\lambda^{\frac{1}{2}}$ is a Salem number, then $\deg \lambda = \deg \lambda^{\frac{1}{2}}$ 
by Lemma \ref{lem:12}. 
\end{proof}

\begin{lemma}  % 16 ( 7.3 )
\label{lem:16}
If $\lambda$ be a Salem number and  $K$ is a subfield of $\rationalnos(\lambda+\lambda^{-1})$, then
$$\deg(\lambda^{\frac{1}{2}}) = \deg_K(\lambda^{\frac{1}{2}})[K:\rationalnos].$$
\end{lemma}
\begin{proof}
As $(\lambda^{\frac{1}{2}}+\lambda^{-\frac{1}{2}})^2 = \lambda +2 + \lambda^{-1}$, we have that 
$$\rationalnos \subseteq K \subseteq \rationalnos(\lambda +\lambda^{-1}) \subseteq 
\rationalnos(\lambda^{\frac{1}{2}}+\lambda^{-{\frac{1}{2}}}).$$
Hence we have that 
$$\rationalnos(\lambda^{\frac{1}{2}}) \subseteq K(\lambda^{\frac{1}{2}}) 
\subseteq \rationalnos(\lambda^{\frac{1}{2}}+\lambda^{-{\frac{1}{2}}})(\lambda^{\frac{1}{2}}) = \rationalnos(\lambda^{\frac{1}{2}}).$$
Thus $\rationalnos(\lambda^{\frac{1}{2}}) = K(\lambda^{\frac{1}{2}})$,
which enables us to deduce exactly as in Lemma \ref{lem:14}.
\end{proof}

Let $\lambda$ be a Salem number, let $K$ be a subfield of $\rationalnos(\lambda +\lambda^{-1})$, 
and let $p(x)$ be the minimal polynomial of $\lambda$ over $K$.  
We say that $\lambda$ is {\it square-rootable over} $K$ if there exist a totally positive element $\alpha$ of $K$ 
and a monic palindromic polynomial $q(x)$, whose even degree coefficients are in $K$ and whose odd degree coefficients 
are in $\sqrt{\alpha}K$, such that $q(x)q(-x) = p(x^2)$.  
We also say that $\lambda$ is {\it square-rootable over} $K$ {\it via} $\alpha$.

\begin{lemma}  % 17 ( 7.4 )
\label{lem:17}
Let $\lambda$ be a Salem number and let $K$ be a subfield of $\rationalnos(\lambda+\lambda^{-1})$. 
Then $\lambda$ is square-rootable over $K$ via a square in $K$ if and only if $\lambda^{\frac{1}{2}}$ 
is a Salem number. 
\end{lemma}
\begin{proof}
Let $p(x)$ be the minimal polynomial of $\lambda$ over $K$. 
First assume that $\lambda^{\frac{1}{2}}$ is a Salem number. 
Let $q(x)$ be the minimal polynomials of $\lambda^{\frac{1}{2}}$ over $K$.  
We have that $\deg \lambda^{\frac{1}{2}} = \deg\lambda$ by Lemma \ref{lem:15}, 
and so $\deg_K(\lambda^{\frac{1}{2}}) = \deg_K(\lambda)$ by Lemmas \ref{lem:14} and \ref{lem:16}.   
Hence $\deg q(x) = \deg p(x)$. 
Now the real roots of $q(x)$ are $\lambda^{\frac{1}{2}}$ and $\lambda^{-\frac{1}{2}}$, 
the degree of $q(x)$ is even, and $q(x)$ is palindromic by Lemma \ref{lem:14.5}. 

As the real roots of $q(x)$ are positive, $q(x) \neq q(-x)$. 
As $\deg q(x)$ is even, $q(-x)$ is monic, and so $q(-x)$ 
is the minimal polynomial of $-\lambda^{\frac{1}{2}}$. 
We conclude that $p(x^2) = q(x) q(-x)$, 
and so $\lambda$ is square-rootable over $K$ via the square $1$. 

Conversely, assume that $\lambda$ is square-rootable over $K$ via a square in $K$. 
Then there exists a monic palindromic polynomial $q(x)$ over $K$ such that $q(x)q(-x) = p(x^2)$. 
By replacing $q(x)$ with $q(-x)$, if necessary, we may assume that $\lambda^{\frac{1}{2}}$ is a root of $q(x)$. 
Hence the minimal polynomial of $\lambda^{\frac{1}{2}}$ over $K$ divides $q(x)$. 
Therefore 
$$\deg_K(\lambda^{\frac{1}{2}})\leq \deg q(x) = \deg p(x) = \deg_K(\lambda).$$
Hence $\deg \lambda^{\frac{1}{2}} \leq \deg \lambda$ by Lemmas \ref{lem:14} and \ref{lem:16}. 
Therefore $\deg \lambda^{\frac{1}{2}} = \deg \lambda$, and either   
$\lambda^{\frac{1}{2}}$ is a Salem number or $\deg \lambda^{\frac{1}{2}} = 2$ by Lemma \ref{lem:15}. 

Assume that $\deg\lambda^{\frac{1}{2}} = 2$.  Then $\deg \lambda = 2$. 
Hence $q(x)$ is the minimal polynomial of $\lambda^{\frac{1}{2}}$ over $K = \rationalnos$. 
As the constant term of $q(x)$ is $1$, the other root of $q(x)$ is $\lambda^{-\frac{1}{2}}$. 
Therefore $\lambda^{\frac{1}{2}}$ is a Salem number with Salem polynomial $q(x)$. 
\end{proof}

\begin{lemma}  % ( 7.5 )
\label{lem:18}
Let $\lambda$ be a Salem number, and let $K$ be a subfield of  $\Q(\lambda + \lambda^{-1})$. 
If $\deg_K(\lambda) = 2$, then $\lambda$ is square-rootable over $\Q(\lambda + \lambda^{-1}) = K$ via $\alpha =  \lambda + \lambda^{-1}+2$. 
\end{lemma}
\begin{proof}
By Lemma \ref{lem:14.5}, the minimal polynomial of $\lambda$ over $K$ is 
$$p(x) = (x-\lambda)(x-\lambda^{-1}) = x^2-(\lambda+\lambda^{-1})x +1.$$
Hence $\lambda + \lambda^{-1} \in K$, and so $K = \Q(\lambda + \lambda^{-1})$. 
Let
$$q(x) = (x-\lambda^{\frac{1}{2}})(x-\lambda^{-\frac{1}{2}}) = x^2-(\lambda^{\frac{1}{2}}+\lambda^{-\frac{1}{2}})x+1.$$
Then $q(x)q(-x) = p(x^2)$ and $q(x) = x^2-\sqrt{\alpha}\,x+1.$
Now $\alpha$ is totally positive, 
since if $\sigma: K \to \realnos$ is a nonidentity embedding, 
then $\sigma(\lambda + \lambda^{-1}) = 2\cos\theta$ for some $\theta \in \R$ with $0 < \theta < \pi$ by 
Lemma \ref{lem:4}. 
Hence $\lambda$ is square-rootable over $K$ via $\alpha$. 
\end{proof}

\begin{theorem} % 7 ( 7.6 ) 
\label{thm:7}
Let $\Gamma \subseteq \Isom(H^n)$ be an arithmetic lattice, with $n$ odd and $n \geq 3$,  
of the simplest type defined over a totally real number field $K$.    
Let $\gamma$ be a hyperbolic element of $\Gamma$, and let $\lambda = e^{2\ell(\gamma)}$.  
Then $\lambda$ is a Salem number which is square-rootable over $K$.
\end{theorem}
\begin{proof}
We may assume that $\Gamma \subseteq \O'(f,\R)$ and $\Gamma$ is  commensurable to $\mathrm{O}'(f,\mathfrak{o}_K)$ 
for some admissible quadratic form $f$ over $K$.  
There exists $B \in \mathrm{GO}(f,K)$ such that $\gamma = \frac{1}{\sqrt{b}}B$
with $b = \mu(B)$ totally positive by Lemmas \ref{lem:7} and \ref{lem:9}. 
Let $\lambda = e^{2\ell(\gamma)}$. 
Then $\lambda = e^{\ell(\gamma^2)}$, and so $\lambda$ is a Salem number 
with $K \subseteq \rationalnos(\lambda+\lambda^{-1})$ by Lemma \ref{lem:10} and Theorem \ref{thm:4}. 
If $b$ is a square in $K$, then $\gamma$ is over $K$ and $\lambda^{\frac{1}{2}}$ is a Salem number 
as in the proof of Theorem \ref{thm:4}(1),  
and so $\lambda$ is square-rootable over $K$ by Lemma \ref{lem:17}. 

For $b$ not a square in $K$, we consider the real quadratic extension $L = K(\sqrt{b})$.
We have that $\deg_L(\lambda^{\frac{1}{2}}) = (1/2)\deg_K(\lambda^{\frac{1}{2}})$ or $\deg_K(\lambda^{\frac{1}{2}})$
depending on whether or not $\sqrt{b}$ is in $K(\lambda^{\frac{1}{2}})$. 
If $\deg\lambda^{\frac{1}{2}}= 2$, then $\deg\lambda = 2$, and so $\lambda$ is square-rootable over $\Q = K$ 
by Lemma \ref{lem:18}.  Hence we may assume $\deg\lambda^{\frac{1}{2}}> 2$. 
If $\deg \lambda^{\frac{1}{2}} = \deg \lambda$, then $\lambda$ is square-rootable over $K$ 
by Lemmas \ref{lem:15} and \ref{lem:17}. 
Hence, we may assume that $\deg \lambda^{\frac{1}{2}} = 2\deg \lambda$.
Then $\deg_K(\lambda^{\frac{1}{2}}) = 2\deg_K(\lambda)$ by Lemmas \ref{lem:14} and \ref{lem:16}. 

Let $p(x)$ be the minimal polynomial of $\lambda$ over $K$, 
and let $q(x)$ be the minimal polynomial of $\lambda^{\frac{1}{2}}$ over $L$.
As $q(x)$ divides the characteristic polynomial of $\gamma$, the real roots of $q(x)$ are $\lambda^\frac{1}{2}$ 
and possibly $\lambda^{-\frac{1}{2}}$, and the complex roots occur in inverse pairs. 
The number field $L$ is totally real, since $b$ is a totally positive element of $K$. 
Hence, the same argument as in the proof of Lemma \ref{lem:14.5} shows that the real roots of $q(x)$ are 
$\lambda^{\frac{1}{2}}$ and $\lambda^{-\frac{1}{2}}$, the degree of $q(x)$ is even, 
and $q(x)$ is palindromic.

Assume that $\deg_L(\lambda^{\frac{1}{2}}) = \deg_K(\lambda^{\frac{1}{2}})$. 
Then $\deg q(x) = 2 \deg p(x)$, and it follows that $q(x) = p(x^2)$. This is a contradiction, since the real roots
of $q(x)$ are positive.  Thus we must have 
$$\deg q(x) = \deg_L(\lambda^{\frac{1}{2}}) = (1/2)\deg_K(\lambda^{\frac{1}{2}}) = \deg_K(\lambda) = \deg p(x).$$
As $\deg q(x)$ is even, $q(-x)$ is monic. Hence $q(-x)$ is the minimal polynomial of
$-\lambda^\frac{1}{2}$ over $L$, and so $q(-x)$ divides $p(x^2)$. It follows that $p(x^2) = q(x) q(-x)$.

Every element of $L$ is of the form $a+c\sqrt{b}$ with $a,c \in K$. 
Let $\tau$ be the automorphism of $L$ over $K$ defined by 
$\tau(a+c\sqrt{b}) =  a -c\sqrt{b}$. 
Now $q(x)q^\tau(x)$ is over $K$, and so $q(x)q^\tau(x) = p(x^2)$, 
since $p(x^2)$ is the minimal polynomial of $\lambda^{\frac{1}{2}}$ over $K$.  
Therefore $q^\tau(x) = q(-x)$. 
Hence the even degree coefficients of $q(x)$ are in $K$ and the odd degree coefficients are in $\sqrt{b}K$. 
Thus $\lambda$ is square-rootable over $K$ via $b$. 
\end{proof}

\begin{theorem} % 8 ( 7.7 )
\label{thm:8}
Let $\lambda$ be a Salem number, let $K$ be a subfield of $\rationalnos(\lambda+\lambda^{-1})$, 
and let $n \geq 3$ be an odd integer with $\deg_K(\lambda) \leq n+1$. 
If $\lambda$ is square-rootable over $K$,  
then there exist an arithmetic lattice $\Gamma \subseteq \Isom(H^n)$ of the simplest type defined over $K$ 
and an orientation preserving hyperbolic element $\gamma$ in $\Gamma$ such that $\lambda^{\frac{1}{2}} = e^{\ell(\gamma)}$. 
\end{theorem}
\begin{proof}
Let $p(x)$ be the minimal polynomial of $\lambda$ over $K$. 
Then $p(\lambda^{-1})=0$ and $m=\deg p(x)$ is even 
by Lemma \ref{lem:14.5}

Assume that $\lambda$ is square-rootable over $K$.  Then there exist a totally positive element $\alpha$ of $K$ 
and a monic palindromic polynomial $q(x)$, whose even degree coefficients are in $K$ 
and whose odd degree coefficients are in $\sqrt{\alpha}K$, such that $q(x)q(-x) = p(x^2)$. 
Then $\deg q(x) = \deg p(x) =m$. 
As $\lambda^{\frac{1}{2}}$ is a root of $p(x^2)$, we have that $\lambda^{\frac{1}{2}}$ is a root of either $q(x)$ or $q(-x)$. 
By replacing $q(x)$ with $q(-x)$, if necessary, we may assume that $\lambda^{\frac{1}{2}}$ is a root of $q(x)$. 

The roots of $p(x^2)$ are of the form $\pm r^{\frac{1}{2}}$ where $r$ is a root of $p(x)$. 
Hence the complex roots of $p(x^2)$ have absolute value 1, and so occur in inverse pairs. 
Therefore the complex roots of $q(x)$ occur in inverse pairs. 
The real roots of $p(x^2)$ are $\lambda^{\pm\frac{1}{2}}$ and $-\lambda^{\pm\frac{1}{2}}$. 
The constant term of $q(x)$ is 1, since $q(x)$ is monic and palindromic. 
Therefore $q(x)$ has $\lambda^{-\frac{1}{2}}$ as a root,  
and so the real roots of $q(x)$ are $\lambda^{\pm\frac{1}{2}}$. 

Assume that $\lambda^{\frac{1}{2}}$ is a Salem number.  Then $\deg \lambda^{\frac{1}{2}} = \deg \lambda$ by Lemma \ref{lem:15},   
and so $\deg_K(\lambda^{\frac{1}{2}}) = \deg_K(\lambda)$ by Lemmas \ref{lem:14} and \ref{lem:16}. 
Hence $\deg_K(\lambda^{\frac{1}{2}}) \leq n+1$.  We have that 
$K \subseteq \rationalnos(\lambda + \lambda^{-1}) \subseteq  \rationalnos(\lambda^{\frac{1}{2}} + \lambda^{-\frac{1}{2}})$. 
Hence there exist an arithmetic group $\Gamma$ of isometries of $H^n$ of the simplest type over $K$ 
and an orientation preserving hyperbolic element $\gamma$ in $\Gamma$ such that $\lambda^{\frac{1}{2}} = e^{\ell(\gamma)}$ 
by Theorem \ref{thm:6}.

Thus we may assume that $\lambda^{\frac{1}{2}}$ is not a Salem number. 
Then $\alpha$ is not a square in $K$ by Lemma \ref{lem:17}. 
Therefore $L = K(\sqrt{\alpha})$ is a quadratic extension of $K$. 
Let $q(x) = a_0+ a_1x+ \cdots + a_mx^m$, and let $\ell = m/2$.   
Then $a_{2j} \in K$ for $j = 0, \ldots, \ell$ and $a_{2j-1} \in \sqrt{\alpha}K$ for $j = 1, \ldots, \ell$. 
Let $b_{2j-1} = a_{2j-1}/\sqrt{\alpha}$ for  $j = 1, \ldots, \ell$. 
Then $b_{2j-1} \in K$ for $j = 1, \ldots, \ell$. 
As $q(\lambda^{\frac{1}{2}}) = 0$, we have that 
$$\sqrt{\alpha}\big(b_1\lambda^{\frac{1}{2}} +b_3\lambda^{\frac{3}{2}} + \cdots + b_{m-1}\lambda^{\frac{m-1}{2}}\big) 
= -\big(a_0 +a_2\lambda +\cdots + a_{m}\lambda^{\ell}\big).$$
Now $a_0 +a_2\lambda +\cdots + a_{m}\lambda^{\ell} \neq 0$, since $\ell < m = \deg p(x)$. 
Therefore $\sqrt{\alpha} \in K(\lambda^{\frac{1}{2}})$, 
and so $L$ is a subfield of $K(\lambda^{\frac{1}{2}})$. 

Next we show that the roots of $p(x^2)$ are simple. 
Assume first that $\deg \lambda^{\frac{1}{2}} = \deg \lambda$. 
As $\lambda^{\frac{1}{2}}$ is not a Salem number, 
$\deg \lambda^{\frac{1}{2}} = 2$ by Lemma \ref{lem:15}. 
Then $\deg \lambda = 2$. 
As $K \subseteq \rationalnos(\lambda+\lambda^{-1}) = \rationalnos$, 
we have that $K = \rationalnos$. 
Therefore $\deg p(x) = 2$, and so the roots of $p(x^2)$ are  $\lambda^{\pm\frac{1}{2}}$ and $-\lambda^{\pm\frac{1}{2}}$ . 
Hence the roots of $p(x^2)$ are simple. 
Now assume that $\deg \lambda^{\frac{1}{2}} \neq \deg \lambda$. 
Then  $\deg \lambda^{\frac{1}{2}} = 2 \deg \lambda$. 
Hence  $\deg_K(\lambda^{\frac{1}{2}}) = 2 \deg_K(\lambda)$ by Lemmas \ref{lem:14} and \ref{lem:16}.
Therefore $p(x^2)$ is the minimal polynomial of $\lambda^{\frac{1}{2}}$ over $K$. 
Hence the roots of $p(x^2)$ are simple. 
In either case, as $q(x)q(-x) = p(x^2)$, the roots of $q(x)$ and $q(-x)$ are simple, and the roots of $q(x)$ are distinct from the roots of $q(-x)$. 

Let $s_1, s_2, \ldots, s_{m-1} = \lambda^{-\frac{1}{2}}, s_m=\lambda^{\frac{1}{2}}$ be the roots of $q(x)$ taken 
with $s_{2j} = \overline{s}_{2j-1}$ of absolute value 1.  
Say $s_{2j} =e^{i\theta_j}$, for $j = 1,\ldots,\ell -1$, and $s_m = \lambda^{\frac{1}{2}} = e^\eta$. 
Then the roots of $q(-x)$ are $-s_1,  \ldots, -s_m$. 
As $s_1, \ldots, s_m, -s_1, \ldots, -s_m$ are the roots of $p(x^2)$, 
the roots of $p(x)$ are $r_k = s^2_k$ for $k = 1, \ldots, m$. 
Now  $r_{2j} = e^{i2\theta_j}$ for $j = 1,\ldots, \ell -1$ and $r_m = \lambda = e^{2\eta}$. 

Let $S = \mathrm{diag}(s_1, s_2, \ldots, s_m)$ and 
let $B$ be the $m\times m$ block diagonal matrix with the first $\ell-1$ blocks 
$$\left(\begin{array}{cr}  1 & -i\\
                                         1& i 
                                     \end{array}\right)\ \ \hbox{and last block} \ \ 
                                     \left(\begin{array}{cr}  1& -1\\
                                      				     1& 1
                                     \end{array}\right).$$
Then $M = B^{-1}SB$ is the block diagonal $m \times m$ matrix with blocks 
$$\left(\begin{array}{cr}  \cos \theta_j &-\sin \theta_j \\
                                     \sin \theta_j & \cos \theta_j
                                     \end{array}\right) 
\ \ \hbox{for}\ \ 1 \leq j < \ell,\ \ \hbox{and} \ \ 
\left(\begin{array}{cc}  \cosh \eta & \sinh \eta \\
                                     \sinh \eta & \cosh \eta
                                     \end{array}\right).$$
The matrix $M$ represents a hyperbolic element of $\mathrm{O}'(m-1,1)$ with characteristic polynomial $q(x)$, 
since the eigenvalues of $M$ are $e^{\pm i\theta_1},\ldots, e^{\pm i\theta_{\ell-1}}$ and $e^{\pm\eta}$. 
Moreover $\det M = 1$, and so $M$ represents an orientation preserving isometry of $H^{m-1}$.  

Let $V = (v_{ij}) = (r_i^{j-1})$ be the Vandermonde matrix for the roots of $p(x)$. 
Then $V$ is invertible, since the roots of $p(x)$ are distinct. 
Let $R = \mathrm{diag}(r_1,\ldots, r_m)$.  
Then $R = S^2$. 
Let $C$ be the companion matrix for $p(x)$. 
Then $C$ is over $\mathfrak{o}_K$ 
and $VC =RV$. 
Let $D = V^{-1}SV$.  Then $D^2 = V^{-1}RV = C$. 
Let $W = B^{-1}V$. Then  
$$WDW^{-1} = (B^{-1}V)(V^{-1}SV)(B^{-1}V)^{-1}= B^{-1}SB = M.$$

Let $J$ be the $m\times m$ matrix $\mathrm{diag}(1,\ldots, 1, -1)$. 
Then $M^tJM = J$. 
Let $A = W^t JW$. 
Then $A$ is a symmetric $m \times m$ matrix, and as in the proof of Theorem \ref{thm:6},  we have that 
$A = (\sum_{k=1}^m r_k^{i-j}/2)$ and $2A$ is over $\mathfrak{o}_K$.  
Now $A$ is the coefficient matrix of a quadratic form $f$ over $K$ in $m$ variables.
If $x \in \realnos^m$, then 
$$f(x) = x^tAx = x^tW^tJWx = (Wx)^tJWx = f_{m-1}(Wx),$$
and so $f$ has signature $(m-1,1)$ and 
$$\mathrm{O}'(f,\realnos) = W^{-1}\mathrm{O}'(m-1,1)W.$$
Now $M \in \mathrm{O}'(m-1,1)$ and $D = W^{-1}MW$.  
Hence $D \in \mathrm{O}'(f,\realnos)$.
The quadratic form $f$ is admissible by the same argument as in the proof of Theorem \ref{thm:6}.  
Note that as $C = D^2$, the matrix $M^2$ plays the role of $M$ in the proof of Theorem \ref{thm:6}.

We next show that the matrix $\sqrt{\alpha}D$ is over $K$ by a Galois group argument. 
Let $\hat K = K(s_1,\ldots,s_m)$ be the splitting field of $p(x^2)$ over $K$, 
and let $G = \mathrm{Gal}(\hat K/K)$. 
As $L$ is a quadratic extension of $K$ contained in $\hat K$, 
there is an index 2 subgroup $H$ of $G$ such that $H = \mathrm{Gal}(\hat K/L)$ by Theorem 3 on p 196 of \cite{Lang}. 
Now $\hat K$ is also the splitting field of $q(x)$ over $L$, and so $H$ is the Galois group of $q(x)$. 
Hence all the elements of $H$ permute the roots $s_1,\ldots, s_m$ of $q(x)$ among themselves.  

Every element of $L$ is of the form $a+c\sqrt{\alpha}$ with $a,c \in K$. 
Let $\tau$ be the automorphism of $L$ over $K$ defined by 
$\tau(a+c\sqrt{\alpha}) =  a -c\sqrt{\alpha}$. 
Then $\tau$ extends to an automorphism $\hat \tau$ of $\hat K$. 
Observe that 
$$\prod_{j=1}^m(x-\hat\tau(s_j)) = q^\tau(x) = q(-x) = \prod_{j=1}^m(x+s_j).$$
Hence we have that $\hat\tau(\{s_1,\ldots,s_m\}) = \{-s_1,\ldots, -s_m\}$. 
Let $\sigma \in G$. 
If $\sigma \in H\hat \tau$, then $\sigma$ extends $\tau$ and $\sigma(\{s_1,\ldots,s_m\}) = \{-s_1,\ldots, -s_m\}$. 
Let $\pi_\sigma$ be the permutation of the indices $1, \ldots, m$ such that $\sigma(s_k) = \pm s_{\pi_\sigma(k)}$ 
for all $k = 1,\ldots, m$, with the plus sign if and only if $\sigma \in H$. 
Then $\sigma(r_k) = \sigma(s^2_k) =s^2_{\pi_\sigma(k)}=r_{\pi_\sigma(k)}$, and so $\sigma$ acts on the roots of $p(x)$ via 
$\pi_\sigma$. 
Hence 
$$V^\sigma = (\sigma(v_{ij})) = (\sigma(r_i^{j-1})) = (r_{\pi_\sigma(i)}^{j-1}) = (v_{\pi_\sigma(i),j}),$$ 
that is, $\sigma$ acts on $V$ by permuting rows via $\pi_\sigma$. 
But then $\sigma$ acts on $V^{-1}$ by permuting columns via $\pi_\sigma$. 
Let $V^{-1} = T = (t_{ij})$.  Then $T^\sigma = (t_{i,\pi_\sigma(j)})$. 

Now $D = V^{-1}SV$, and so the $ij$-entry of $D$ is $d_{ij} = \sum_{k=1}^mt_{ik}s_kv_{kj}$. 
Then 
$$\sigma(d_{ij}) = \sum_{k=1}^mt_{i,\pi_\sigma(k)}(\pm s_{\pi_\sigma(k)})v_{\pi_\sigma(k),j}=\pm d_{ij}$$
with the plus sign if and only if $\sigma \in H$.  Hence $(\sqrt{\alpha}D)^\sigma = \sqrt{\alpha}D$ for all $\sigma \in G$, 
and therefore $\sqrt{\alpha}D$ is over $K$. 

Let $m$ be a positive integer such that $E = m\sqrt{\alpha}D$ is over $\mathfrak{o}_K$. 
Then $\det E = m^m\alpha^{\frac{m}{2}}$. 
Let $\varepsilon = m^2\alpha$.  Then $\varepsilon^{\frac{m}{2}} = \det E$ is in $\mathfrak{o}_K$. 
Hence $\varepsilon \in \mathbb{A}\cap K = \mathfrak{o}_K$. 
We have that $\frac{1}{\sqrt{\varepsilon}}E= D$,  
and so $E^2 = \varepsilon D^2 = \varepsilon C$. 

Let $\Phi$ be the congruence $\varepsilon$ subgroup of $\mathrm{O}'(f,\mathfrak{o}_K)$. 
Then $\Phi$ is a normal subgroup of $\mathrm{O}'(f,\mathfrak{o}_K)$ of finite index, 
since the quotient ring $\mathfrak{o}_K/(\varepsilon)$ is finite.

Let $\Psi$ be the subgroup of $\mathrm{O}'(f,\realnos)$ generated by $C$ and the elements of $\Phi$  and  $D\Phi D^{-1}$. 
We claim that $\Psi$ is a subgroup of $\mathrm{O}'(f,\mathfrak{o}_K)$. 
First of all, $C \in \mathrm{O}'(f,\mathfrak{o}_K)$. 
Let $X \in \Phi$.  Then $DXD^{-1} = DXDC^{-1}$. 
Now $EXE$ is congruent modulo $\varepsilon$ to $E^2 = \varepsilon C$, 
and so $DXD = EXE/\varepsilon$ is over $\mathfrak{o}_K$. 
Hence $DXD^{-1} \in \mathrm{O}'(f,\mathfrak{o}_K)$. 
Therefore $D\Phi D^{-1}$ is a subgroup of $\mathrm{O}'(f,\mathfrak{o}_K)$, 
and so $\Psi$ is a subgroup of $\mathrm{O}'(f,\mathfrak{o}_K)$. 

Let $\Delta$ be the subgroup of $\mathrm{O}'(f,\realnos)$ generated by $D$ and the elements of $\Psi$.
We claim that $\Psi$ is a normal subgroup of $\Delta$. 
It suffices to show that $D\Psi D^{-1} = \Psi$. 
As $C=D^2$, we have that $DCD^{-1} = C$. 
Now $D(D\Phi D^{-1})D^{-1}= C\Phi C^{-1} = \Phi$, since $\Phi$ is a normal subgroup of $\mathrm{O}'(f,\mathfrak{o}_K)$. 
Therefore $D$ conjugates the set of generators of $\Psi$ to itself. 
Hence $D\Psi D^{-1} = \Psi$, and so $\Psi$ is a normal subgroup of $\Delta$.

Now $D$ is not over $K$, and so $D$ is not in $\Psi$. 
Hence $\Psi$ is a subgroup of index 2 in $\Delta$, since $D^2 = C$ is in $\Psi$. 
Moreover  $\Delta \cap \mathrm{O}'(f,\mathfrak{o}_K) = \Psi$. 
We have that $\Psi$ has finite index in $\mathrm{O}'(f,\mathfrak{o}_K)$. 
Therefore $\Delta$ is commensurable to $\mathrm{O}'(f,\mathfrak{o}_K)$. 
Hence $\Gamma = W\Delta W^{-1}$ is an arithmetic group of isometries of $H^{m-1}$ of the simplest type defined over $K$.  
We have that $\gamma = M = WDW^{-1}$ is an orientation preserving hyperbolic element of $\Gamma$ 
such that $\lambda^{\frac{1}{2}} = e^{\ell(\gamma)}$.  
If $m = n + 1$, we are done, so assume that $m < n+1$. 

Consider the following $2\times 2$ matrices
$$D_1=\left(\begin{array}{cc} 0 & -\sqrt{\alpha} \\ \frac{1}{\sqrt{\alpha}} & 0 \end{array}\right), \ 
C_1=\left(\begin{array}{cc} -1 & 0 \\ 0 & -1 \end{array}\right), \ 
A_1=\left(\begin{array}{cc}  1 & 0 \\ 0 & \alpha  \end{array}\right), $$ 
$$W_1=\left(\begin{array}{cc}  1 & 0 \\ 0 & \sqrt{\alpha}  \end{array}\right), \ 
J_1=\left(\begin{array}{cc}  1 & 0 \\ 0 & 1 \end{array}\right), \ 
M_1=\left(\begin{array}{cc}  0 & -1 \\ 1 & 0  \end{array}\right).$$
If $X$ is an $m \times m$ matrix, let $\hat X$ be the block diagonal $(n+1)\times (n+1)$ matrix with $(n+1-m)/2$ blocks of the form $X_1$ and final block $X$. 
Then $\hat D^2 = \hat C$. 
Now $\hat A$ is a symmetric $(n+1)\times (n+1)$ matrix such that $\hat A = \hat W^t \hat J \hat W$. 
Hence $\hat A$ is the coefficient matrix of a quadratic form $\hat f$ over $K$ in $n+1$ variables  
of signature $(n,1)$ with $\mathrm{O}'(\hat f, \realnos) = \hat W^{-1}\mathrm{O}^+(n,1)\hat W$. 
Moreover $\hat f$ is admissible, since $f$ is admissible and $\alpha$ is totally positive. 
We have that $\hat D^t \hat A \hat D = \hat A$ and $\hat W\hat D \hat W^{-1} = \hat M$, 
and so $\hat D$ is in $\mathrm{O}'(\hat f, \realnos)$. 
Moreover $\sqrt{\alpha}\hat D$ is over $K$. 

Let $\hat m$ be a positive integer such that $\hat m\sqrt{\alpha}\hat D$ is over $\mathfrak{o}_K$, 
and let $\hat \varepsilon = \hat m^2\alpha$. Then as above, $\hat \varepsilon\in \mathfrak{o}_K$. 
Let $\hat \Phi$ be the congruence $\hat \varepsilon$ subgroup of $\mathrm{O}'(\hat f, \mathfrak{o}_K)$, 
let $\hat \Psi$ be the subgroup of $\mathrm{O}'(\hat f,\realnos)$ generated by $\hat C$ and the elements of $\hat\Phi$ 
and $\hat D\hat \Phi \hat D^{-1}$, and let $\hat \Delta$ be the subgroup of $\mathrm{O}'(\hat f,\realnos)$ generated by $\hat D$ and the elements of $\hat\Psi$.  
Then as above, $\hat\Delta$ is commensurable to $\mathrm{O}'(\hat f, \mathfrak{o}_K)$. 
Hence $\hat \Gamma = \hat W\hat \Delta\hat W^{-1}$ is an arithmetic group of isometries of $H^n$ of the simplest type 
defined over $K$. 
We have that $\hat \gamma = \hat M = \hat W\hat D\hat W^{-1}$ is an orientation preserving hyperbolic element of $\hat \Gamma$ 
such that $\lambda^{\frac{1}{2}} = e^{\ell(\hat\gamma)}$. 
\end{proof}

The next corollary is an enhanced version of Corollary \ref{cor:3}. 

\begin{corollary}  % 7 ( 7.8 )
\label{cor:7}
Let $\lambda$ be a Salem number. 
Then for each odd integer $n \geq 3$, 
there exist an arithmetic lattice $\Gamma \subseteq \Isom(H^n)$ of the simplest type 
defined over $\rationalnos(\lambda+\lambda^{-1})$ and an orientation preserving hyperbolic element $\gamma$ of $\Gamma$  
such that $\gamma^4$ is a hyperbolic translation 
and $\lambda^{\frac{1}{2}} = e^{\ell(\gamma)}$. 
\end{corollary}
\begin{proof} 
Let $K=\rationalnos(\lambda + \lambda^{-1})$.  Then $\deg_K(\lambda) = 2$ by Theorem \ref{thm:4}(2).  
Hence $\deg_K(\lambda) \leq n+1$ for $n \geq 1$. 
Moreover $\lambda$ is square-rootable over $K$ by Lemma \ref{lem:18}. 
Therefore $\lambda^{\frac{1}{2}} = e^{\ell(\gamma)}$ for some $\gamma \in \Gamma$ given by Theorem \ref{thm:8}. 
Moreover $\gamma^4$ is a hyperbolic translation by the proof of Theorem \ref{thm:8}.    
\end{proof}

The next corollary follows from Theorems \ref{thm:7} and \ref{thm:8}. 

\begin{corollary}  % 8 ( 7.10 )
\label{cor:8}
Let $\Gamma \subseteq \Isom(H^n)$ be an arithmetic lattice of the simplest type defined over $\rationalnos$, with $n$ odd, 
and let $C$ be a closed geodesic in $H^n/\Gamma$. 
Then $\mathrm{length}(C) \geq c_n$,  
and this lower bound is sharp for each odd integer $n \geq 3$. 
\end{corollary}

Corollary \ref{cor:4} follows from Corollary \ref{cor:8} once we have a sharp non-cocompact example for $n = 3$, 
since all arithmetic groups of isometries of $H^n$ of the simplest type over $\rationalnos$ are not cocompact when $n > 3$. 
Such an example will be described in the next section. 

%%%%%%%%%%%%%%%%%%%%%%%%%

\section{The values of $c_n$ for odd $n \leq 19$} % 9
\label{sec:9}

We begin by considering some necessary conditions for square-rootability in Lemma \ref{lem:19}, and some sufficient 
conditions for square-rootability in Lemma \ref{lem:20}. 

\begin{lemma}  % 19 ( 8.1 )
\label{lem:19}
Let $\lambda$ be a Salem number, let $K$ be a subfield of $\rationalnos(\lambda +\lambda^{-1})$, 
let $p(x)$ be the minimal polynomial of $\lambda$ over $K$, and let $m = \deg p(x)$.  
Suppose that $\lambda$ is square-rootable over $K$ via $\alpha$ in $K$. 
\begin{enumerate}
\item If $m \equiv 0 \ \mathrm{mod} \ 4$, then $p(-1)$ is a square in $\mathfrak{o}_K$. 
\item If $m \equiv 2 \ \mathrm{mod} \ 4$, then there exists $k \in K^\times$ such that $p(-1) = \alpha k^2$. 
\end{enumerate}
\end{lemma}
\begin{proof}
There exists a monic palindromic polynomial $q(x)$, whose even degree coefficients are in $K$ and whose 
odd degree coefficients are $\sqrt{\alpha}\,K$, such that $q(x)q(-x) = p(x^2)$. 
Hence $p(-1) = q(i)q(-i)$. 
We have that 
$$q(-i) = q(1/i) = i^{-m}i^mq(1/i) = i^{-m}q(i).$$
Hence we have that $p(-1) = i^{-m}q(i)^2$. 

(1)  Assume that $m \equiv 0 \ \mathrm{mod} \ 4$.  Then $i^m=1$.  Hence $p(-1) = q(i)^2$. 
If $k$ is an odd positive integer, then $i^{m-k} = i^{-k} = (-i)^k=  -i^k$. 
Hence the odd degree terms of $q(x)$ cancel in the evaluation of $q(i)$. 
The roots of $q(x)$ are in $\mathbb{A}$, and so the even degree coefficients of $q(x)$ 
are in $\mathbb{A}\cap K = \mathfrak{o}_K$. 
Therefore $q(i) \in  \mathfrak{o}_K$. Hence $p(-1)$ is a square in $\mathfrak{o}_K$.

(2) Assume that $m \equiv 2 \ \mathrm{mod} \ 4$. Then $i^m = -1$.  Hence $p(-1) = -q(i)^2$. 
If $k$ is an even nonnegative integer, then $i^{m-k} = -i^{-k} = -(-i)^k=  -i^k$. 
Therefore the even degree terms of $q(x)$ cancel in the evaluation of $q(i)$. 
Hence there exists $k \in K$ such that $q(i) = \sqrt{\alpha}\,k\, i$. 
Therefore $p(-1) = \alpha k^2$.  As $p(-1) \neq 0$, we have that $k \neq 0$. 
\end{proof}

\begin{lemma} % 20 (8.2 )
\label{lem:20}
Let $\lambda$ be a Salem number, let $K$ be a subfield of $\rationalnos(\lambda +\lambda^{-1})$, 
and let $p(x)$ be the minimal polynomial of $\lambda$ over $K$.  
\begin{enumerate}
\item If $p(x) = x^4 + ax^3 + bx^2+ ax+1$, then $\lambda$ is square-rootable over $K$ if and only if 
there is a positive element $k$ of $\mathfrak{o}_K$ such that $p(-1) = k^2$ 
and $4-a \pm 2k$ is a totally positive element of $K$, 
in which case $\lambda$ is square-rootable over $K$ via $4-a \pm 2k$. 
\item If $\deg p(x) = 4$ and $K = \rationalnos$, then $\lambda$ is square-rootable over $K$ 
if and only if $p(-1)$ is a square in $\integers$. 
\end{enumerate}
\end{lemma}
\begin{proof}
(1) Suppose that $\lambda$ is square-rootable over $K$ via $\alpha$. 
Then $p(-1) = k^2$ with $k \in \mathfrak{o}_K$ by Lemma \ref{lem:19}(1). 
Now $k \neq 0$, since $p(-1) \neq 0$.  
By replacing $k$ with $-k$, if necessary, we may assume that $k > 0$. 
Let 
$$q(x) = x^4 + c x^3 + d x^2 + c x +1$$  
such that $q(x)q(-x) = p(x^2)$ 
with $c = \sqrt{\alpha} \ell$ for some $\ell \in K$ and $d \in K$. 
Then $2d -c^2 = a$ and $2 -2c^2 +d^2 = b$. 
Hence $c^2 = 2d -a$, and so $2-4d + 2a + d^2 = b$. 
Therefore 
$$d = 2 \pm \sqrt{2-2a +b} = 2 \pm \sqrt{p(-1)} = 2 \pm k$$
and 
$$c = \pm \sqrt{2d -a} = \pm \sqrt{4 \pm 2k -a}.$$
As $4-a\pm 2k = c^2 = \alpha\ell^2$, we have that $4-a\pm 2k$ is totally positive. 

Conversely, if $k$ is a positive element of $\mathfrak{o}_K$ such that $p(-1) = k^2$ and 
$4-a\pm 2k$ is totally positive, then we can solve for $c$ and $d$ from the above equations 
and deduce that $\lambda$ is square-rootable over $K$ via $4-a\pm 2k$. 

(2)  By (1) it suffices to show that if $\lambda$ is square-rootable over $K = \rationalnos$, then $4-a+2k$ is positive. 
Let $\lambda^{\pm 1}, \mu^{\pm 1}$ be the roots of $p(x)$. 
Then 
$-a = \lambda+\lambda^{-1}+\mu+\mu^{-1}.$
We have that $\lambda+\lambda^{-1} > 2$ and $\mu + \mu^{-1} = 2\cos\theta$ for some real number $\theta$, 
and so $-a > 0$.  Therefore $4-a+2k > 0$. 
\end{proof}

Recall that for each odd positive integer $n$, we defined
\begin{eqnarray*}
c_n & = & \mathrm{min}\{\textstyle{\frac{1}{2}}\log \lambda : \lambda\  \hbox{is a Salem number with} \ \deg \lambda \leq n+1, \\
& & \hspace{.32in} \hbox{which is square-rootable over}\  \mathbb{Q}\}. 
\end{eqnarray*}
Let $\lambda_{m,\ell}$ be the $\ell$th largest Salem number of degree $m$ listed in \cite{Moss}. 
It follows from Lemma \ref{lem:18} that 
$$c_1  =  \textstyle{\frac{1}{2}}\log \lambda_{2,1} = 0.481211825 \ldots\ .$$

The smallest Salem number of degree 4 with Salem polynomial $p(x)$ such that $p(-1)$ is a square in $\integers$ is 
$$\lambda_{4,6} = \frac{1}{4}\left(1 + \sqrt{21} + \sqrt{2 \big(3 + \sqrt{21}\big)}\right) = 2.3692054071\ldots\ $$
with Salem polynomial 
$$p(x) = x^4 - x^3 - 3x^2-x+1.$$ 
We have that $p(-1) = 1$, and so $\lambda_{4,6}$ is square-rootable over $\rationalnos$ 
via $3$ and $7$ by Lemma \ref{lem:20}. 
Hence we have that 
$$c_3 =  \textstyle{\frac{1}{2}}\log \lambda_{4,6} = 0.4312773138 \ldots \ . $$

To finish the proof of Corollary \ref{cor:4}, we need a non-cocompact arithmetic group $\Gamma$ of isometries of $H^3$ and a hyperbolic element $\gamma$ of $\Gamma$ 
such that $\lambda_{4,6} = e^{2\ell(\gamma)}$. 
The proof of Theorem \ref{thm:8} yields an arithmetic group $\Gamma$ of isometries of $H^3$ of the simplest type over $\rationalnos$ 
and a hyperbolic element $\gamma$ of $\Gamma$ 
such that $\lambda_{4,6} = e^{2\ell(\gamma)}$. 
A conjugate of $\Gamma$ is commensurable to $\mathrm{O}'(f,\integers)$ where the quadratic form $f$ has coefficient matrix 
$$A = \frac{1}{2}\left(\begin{array}{cccc} 4 & 1 & 7 & 13 \\ 1 & 4 & 1 & 7 \\ 7 & 1 & 4 & 1 \\ 13 & 7 & 1 & 4 \end{array}\right).$$
The only solution of $f(x) \equiv 0\ \mathrm{mod}\ 7$ with $x \in \integers^4$ is $x \equiv (0,0,0,0) \ \mathrm{mod}\ 7$. 
Hence, by a descent argument, the only solution of $f(x) = 0$ with $x \in \integers^4$ is $x = (0,0,0,0)$, 
and so the only solution of $f(x) = 0$ with $x \in \rationalnos^4$ is $x = (0,0,0,0)$. 
Therefore $\mathrm{O}'(f,\integers)$ is cocompact, and so $\Gamma$ is cocompact, 
which is not what we need. 

Let $K = \rationalnos(\sqrt{-3})$ and let $\omega = (1+i\sqrt{3})/2$.  
Then $\mathfrak{o}_K = \{a + b\omega: a, b \in \integers\}$. 
The group $\mathrm{PSL}(2,\mathfrak{o}_K)$ is a non-cocompact arithmetic group of isometries 
of the upper half-space model of hyperbolic 3-space which contains a loxodromic hyperbolic element $\eta$ 
represented by the matrix
$$\left(\begin{array}{cc} 0 & 1 \\ -1 & \omega \end{array}\right).$$
We have that $\lambda_{4,6} = e^{\ell(\eta)}$ (see\,\cite{N-R} \S 4). 
By Theorem 4.11 of \cite{N-R} there is a subgroup $\Gamma$ of $\mathrm{PSL}(2,\complexnos)$ 
that is commensurable to $\mathrm{PSL}(2,\mathfrak{o}_K)$ and a hyperbolic element $\gamma$ of $\Gamma$ 
such that $\gamma^2 = \eta$, and so $\lambda_{4,6} = e^{2\ell(\gamma)}$. 
Thus the bound $c_n$ is sharp in Corollary \ref{cor:4} for $n = 3$. 

In order to find the values of $c_n$ for $n > 4$, we need to determine when a Salem number of degree greater than 4 is square-rootable over $\rationalnos$. 
In this regard, the necessary conditions in Lemma \ref{lem:19} are useful. 
In practice, we used a more systematic method which we describe next. 

Let $p(x)$ be a Salem polynomial for a Salem number $\lambda$ of degree $m > 4$, and let $\ell = m/2$. 
Let $r_1, r_1^{-1}, \ldots, r_\ell, r_\ell^{-1}$ be the roots of $p(x)$ with $r_1 = \lambda$.  
Choose complex numbers $s_1, s_2, \ldots, s_\ell$ so that $s_j^2 = r_j$ for each $j$ and $s_1 = \lambda^{\frac{1}{2}}$. 
There are two choices for each $s_{j}$ with $1 < j \leq \ell $ and so there are a total of $2^{\ell-1}$ choices. 
Let 
$$q(x) = (x^2-(s_1+s_1^{-1})x +1) \cdots(x^2-(s_\ell+s_\ell^{-1})x +1).$$
Then we have that $q(x)q(-x) = p(x^2)$. 
In order for $\lambda$ to be square-rootable over $\rationalnos$, 
the even degree coefficients of $q(x)$ must be integers and the squares of the odd degree coefficients of $q(x)$ 
must be integers with the same square-free part for some choice of $s_1, \ldots, s_\ell$. 
These conditions can be checked numerically. 

We determined that the smallest Salem number of degree 6 that is square-rootable over $\rationalnos$ is 
$$\lambda_{6,4} = 1.5823471836 \ldots$$
with Salem polynomial 
$$p(x) = x^6  - x^4 - 2 x^3 - x^2 + 1$$ 
and 
$$q(x) = x^6 -\sqrt{2}\, x^5 + x^4 -\sqrt{2}\, x^3 + x^2 -\sqrt{2}\, x + 1.$$
Hence we have that
$$c_5 =  \textstyle{\frac{1}{2}}\log \lambda_{6,4} = 0.2294546519 \ldots\ .$$

The smallest Salem number of degree 8 that is square-rootable over $\rationalnos$ is 
$\lambda_{8,8} = \lambda_{8,1}^2$.  As $c_5 < b_8$, we have that $c_5 = c_7$. 

The smallest Salem number of degree 10 that is square-rootable over $\rationalnos$ is 
$\lambda_{10,8} = \lambda_{10,1}^2$. 
Hence we have that 
$$c_9 = b_{10} = 0.1623576120\ldots\ .$$ 

The smallest Salem number of degree 12 that is square-rootable over $\rationalnos$ is 
$\lambda_{12,16} = \lambda_{12,1}^2$. 
Hence we have that $c_{11} = b_{10}$. 

The smallest Salem number of degree 14 that is square-rootable over $\rationalnos$ is 
$\lambda_{14,17} = \lambda_{14,1}^2$. 
Hence we have that $ c_{13} = b_{10}$. 

The smallest Salem number of degree 16 that is square-rootable over $\rationalnos$ is 
$$\lambda_{16, 23} = 1.4908316618 \ldots $$
with Salem polynomial 
$$p(x) = x^{16}-x^{14}-x^{12}-2x^{11}-x^8-2x^5-x^4-x^2+1$$ 
and 
$$q(x) = x^{16} - \sqrt{2}\,x^{15}+ x^{14}- \sqrt{2}\,x^{13} + x^{12}-x^8+x^4- \sqrt{2}\,x^{3}+ x^2 - \sqrt{2}\,x +1.$$
We have that $\textstyle{\frac{1}{2}}\log \lambda_{16,23} = 0.19966 \ldots\ $, 
and so we have that $c_{15} = b_{10}$. 

The smallest Salem number of degree 18 that is square-rootable over $\rationalnos$ is 
$\lambda_{18,22} = \lambda_{18,1}^2$. 
Hence we have that $c_{17} = b_{10}$. 

The smallest Salem number of degree 20 that is square-rootable over $\rationalnos$ is 
$\lambda_{20, 74} = \lambda_{20,1}^2$. 
Hence we have that $c_{19} = b_{10}$. 

%%%%%%%%%

\section{An example with $K$ an intermediate field} %10
\label{sec:10}

All the examples of a Salem number $\lambda$ that is square-rootable over $K$  
that we have considered so far have been with $K = \rationalnos$ or $\rationalnos(\lambda+\lambda^{-1})$. 
In this section, we consider an example with $K$ an intermediate field between $\rationalnos$ and 
$\rationalnos(\lambda+\lambda^{-1})$. 
Consider the polynomial 
$$q(x) = x^4 -4 x^3 - 4x^2 + 4x +1.$$
The polynomial $q(x)$ is irreducible over $\integers$ and has three roots that lie in the open interval $(-2,2)$ 
and one root that is greater than 2. 
Hence $q(x)$ is the trace polynomial of the Salem polynomial 
$$p(x) =  x^4q(x+ x^{-1}) = x^8 - 4x^7-8x^5-x^4-8x^3-4x+1.$$
The corresponding Salem number has value $\lambda = 4.43861 \ldots,$ 
and $q(x)$ is the minimal polynomial of 
$$\lambda+\lambda^{-1}=1+\sqrt{3}+\sqrt{2+\sqrt{3}} = 4.6639\ldots\ .$$

Now $\lambda^{1/2}= 2.1068\ldots $ has minimal polynomial 
$p(x^2)$, since $p(x^2)$ is irreducible over $\integers$,   
and so $\lambda^{1/2}$ is not a Salem number by Lemma \ref{lem:12}.
The polynomial
$$g(x)=q(x^2-2)=x^8-12x^6+44x^4-60x^2+25$$
is also irreducible over $\integers$. 
As 
$$(\lambda^{\frac{1}{2}}+\lambda^{-\frac{1}{2}})^2 = \lambda+\lambda^{-1}+2,$$ 
we have that $g(x)$ is the minimal polynomial of 
$$\lambda^{\frac{1}{2}}+\lambda^{-\frac{1}{2}}=\sqrt{3+\sqrt{3}+\sqrt{2+\sqrt{3}}}= 2.58145\ldots\ .$$

Now we have the factorizations

\begin{eqnarray*}
q(x)&=&\left(x^2-\big(2+\sqrt{2}\big)x-\big(3+2\sqrt{2}\big)\right) 
\left(x^2-\big(2-\sqrt{2}\big)x-\big(3-2\sqrt{2}\big)\right)\\
&=&\left(x^2-\big(2+2\sqrt{3}\big)x+\big(2+\sqrt{3}\big)\right)
\left(x^2-\big(2-2\sqrt{3}\big)x+\big(2-\sqrt{3}\big)\right)\\
&=&\left(x^2-\big(2+\sqrt{6}\big)x-1\right)
\left(x^2-\big(2-\sqrt{6}\big)x-1\right)
\end{eqnarray*}
with $\lambda+ \lambda^{-1}$ a root of the first factor in each case.  
This implies that $\sqrt{2}, \sqrt{3}, \sqrt{6} \in \rationalnos(\lambda+\lambda^{-1})$ 
and that $\rationalnos(\lambda+\lambda^{-1})$ is the splitting field of $q(x)$. 
The polynomial $q(x)$ was carefully chosen so that its Galois group is a Klein four group. 
Hence the intermediate fields between $\rationalnos$ and $\rationalnos(\lambda+\lambda^{-1})$ 
are $\rationalnos(\sqrt{2})$, $\rationalnos(\sqrt{3})$ and $\rationalnos(\sqrt{6})$ corresponding 
to the three subgroups of $\mathrm{Gal}(q(x))$ of index 2 by the fundamental theorem of Galois theory.  

Let $K$ be one of the fields $\rationalnos(\sqrt{2})$, $\rationalnos(\sqrt{3})$ or $\rationalnos(\sqrt{6})$, 
and let $q_K(x)$ be the first factor of the above factorization of $q(x)$ over $K$. 
Then the minimal polynomial of $\lambda$ over $K$ is 
$$p_K(x) = x^2q_K(x+x^{-1}).$$
If $K = \rationalnos(\sqrt{2})$, then $p_K(-1) = q_K(-2) = 5$, which is not a square in $\mathfrak{o}_K$, 
and so $\lambda$ is not square-rootable over $\rationalnos(\sqrt{2})$ by Lemma \ref{lem:19}(1). 
If $K = \rationalnos(\sqrt{3})$, then $p_K(-1) = q_K(-2) = 10 + 5\sqrt{3}$, which is not a square in $\mathfrak{o}_K$, 
and so $\lambda$ is not square-rootable over $\rationalnos(\sqrt{3})$ by Lemma \ref{lem:19}(1). 

Now suppose that $K = \rationalnos(\sqrt{6})$, then 
$$p_K(-1) = q_K(-2) = 7+2\sqrt{6} = \big(1+\sqrt{6}\big)^2.$$
We have that 
$$p_K(x) = x^4-\big(2+\sqrt{6}\big)x^3+x^2-\big(2+\sqrt{6}\big)x+1.$$
Then $\lambda$ is square-rootable over $K$ 
via $\alpha = 4 -\sqrt{6}$ and $\beta = 8 +3\sqrt{6}$ by Lemma \ref{lem:20}(1).  
Note that $\alpha(5+2\sqrt{6}) = \beta$ with $5+2\sqrt{6}$ the fundamental unit of $\o_K$. 

Let $h(x) = x^4-\big(6-\sqrt{6})x^2 + \big(7-2\sqrt{6}\big)$. 
Then we have the factorizations
\begin{eqnarray*}
g(x)&=&\left(x^2-\sqrt{\alpha}\,x-\big(1+\sqrt{6}\big)\right)
\left(x^2+\sqrt{\alpha}\,x-\big(1+\sqrt{6}\big)\right) h(x) \\
&=&\left(x^2-\sqrt{\beta}\,x+\big(1+\sqrt{6}\big)\right)
\left(x^2+\sqrt{\beta}\,x+\big(1+\sqrt{6}\big)\right) h(x)
\end{eqnarray*}
with $\lambda^{1/2}+\lambda^{-1/2}$ a root of the first factor in each case. 
This implies that $\sqrt{\alpha}, \sqrt{\beta} \in \rationalnos(\lambda^{\frac{1}{2}}+\lambda^{-\frac{1}{2}})$. 
The lattice of considered subfields of ${\mathbb Q}\left(\lambda^{1/2}+\lambda^{-1/2}\right)$ is shown
in Figure 1, with each line indicating a degree 2 extension. 

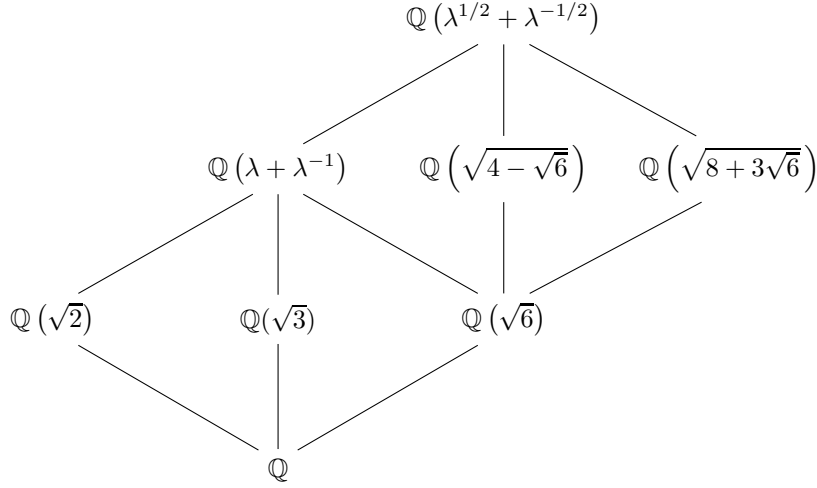
\begin{figure}\label{fig1}
\begin{tikzpicture}
  [node distance=2cm and 3cm,rect/.style={rectangle}]
  \node[rect] (F0) {$\mathbb Q$};
  \node[rect] (F2) [on grid,above=of F0] {${\mathbb Q}(\sqrt{3})$};
  \node[rect] (F1) [on grid,left=of F2] {${\mathbb Q}\left(\sqrt{2}\right)$};
  \node[rect] (F3) [on grid,right=of F2] {${\mathbb Q}\left(\sqrt{6}\right)$};
  \node[rect] (F4) [on grid,above=of F2] {${\mathbb Q}\left(\lambda+\lambda^{-1}\right)$};
  \node[rect] (F5) [on grid,right=of F4] {${\mathbb Q}\left(\sqrt{4-\sqrt{6}}\,\right)$};
  \node[rect] (F6) [on grid,right=of F5] {${\mathbb Q}\left(\sqrt{8+3\sqrt{6}}\,\right)$};
  \node[rect] (F7) [on grid,above=of F5] {${\mathbb Q}\left(\lambda^{1/2}+\lambda^{-1/2}\right)$};
  \draw (F0.135) -- (F1.-45);
  \draw (F0.90) -- (F2.-90);
  \draw (F0.45) -- (F3.-135);
  \draw (F1.45) -- (F4.-135);
  \draw (F2.90) -- (F4.-90);
  \draw (F3.135) -- (F4.-45);
  \draw (F3.90) -- (F5.-90);
  \draw (F3.45) -- (F6.-130);
  \draw (F4.45) -- (F7.-135);
  \draw (F5.90) -- (F7.-90);
  \draw (F6.135) -- (F7.-45);
\end{tikzpicture}
\caption{Lattice of considered subfields of $\rationalnos(\lambda^{\frac{1}{2}}+\lambda^{-\frac{1}{2}})$} 
\end{figure}

Let $r_1, r_2, r_3 = \lambda^{-1}$, and $r_4=\lambda$, be the four roots of $p_K(x)$. 
Let 
$$A = \left(\sum_{k=1}^4 r_k^{i-j}/2\right).$$
The matrix $A$ is a symmetric function of the roots of $p_K(x)$, and so
the entries of $A$ can expressible in terms of the coefficients of $p_K(x)$ by Newton identities, 
which works out to be
$$A=\frac{1}{2}\left(\begin{matrix}
4&2+\sqrt{6}&8+4\sqrt{6}&44+18\sqrt{6}\\
2+\sqrt{6}&4&2+\sqrt{6}&8+4\sqrt{6}\\
8+4\sqrt{6}&2+\sqrt{6}&4&2+\sqrt{6}\\
44+18\sqrt{6}&8+4\sqrt{6}&2+\sqrt{6}&4
\end{matrix}\right).$$
The matrix $A$ has signature $(3,1)$, and the automorphism of $K$ taking $\sqrt{6}$ to $-\sqrt{6}$ takes $A$ to a positive definite matrix, and so the corresponding quadratic form $f(x) = x^tAx$ over $K$ is admissible. 

Let $L = K(\sqrt{\alpha}) = \rationalnos(\sqrt{\alpha})$. 
The minimal polynomial of $\lambda^{\frac{1}{2}}+\lambda^{-\frac{1}{2}}$ over $L$ is 
$$g_L(x) = x^2-\sqrt{\alpha}\,x-(1+\sqrt{6}).$$
The minimal polynomial of $\lambda^{\frac{1}{2}}$ over $L$ is 
$$q_L(x) = x^2g_L(x+ x^{-1}) = x^4 -\sqrt{\alpha}\,x^3 + (1-\sqrt{6})x^2-\sqrt{\alpha}\,x+1,$$
and $q_L(x)q_L(-x) = p_K(x^2)$ showing that $\lambda$ is square-rootable over $K$ via $\alpha$. 

Let $s_1, s_2, s_3=\lambda^{-1/2}$, and $s_4=\lambda^{1/2}$ be the roots of $q_L(x)$, taken in order so that $s_k^2=r_k$ each $k$.
Let $V$ be the Vandermonde matrix $(r_i^{j-1})$, and let
$$D =V^{-1}{\rm diag}(s_1,s_2,s_3,s_4)V.$$
Then we find that 
$$D = \frac{1}{5\sqrt{\alpha}}\left(\begin{matrix}
6-\sqrt{6}&-1+\sqrt{6}&1-\sqrt{6}&-6+\sqrt{6}\\
3-3\sqrt{6}&2-2\sqrt{6}&3+2\sqrt{6}&7+3\sqrt{6}\\
3+2\sqrt{6}&2-2\sqrt{6}&3-3\sqrt{6}&-3+3\sqrt{6}\\
1-\sqrt{6}&-1+\sqrt{6}&6-\sqrt{6}&9+\sqrt{6}
\end{matrix}\right).
$$
Then $D \in \mathrm{O}'(f,\realnos)$ and $D^2 = C$ with $C$ the companion matrix of $p_K(x)$ 
given by
$$C=\left(\begin{matrix}
0&0&0&-1\\
1&0&0&2+\sqrt{6}\\
0&1&0&-1\\
0&0&1&2+\sqrt{6}
\end{matrix}\right).$$ 
The matrices $E=5\sqrt{\alpha}D$ and $C$ are over $\mathfrak{o}_K$, 
and so $C \in \mathrm{O}'(f,\mathfrak{o}_K)$. 
Let $\varepsilon = 5^2\alpha$. 
As in the proof of Theorem \ref{thm:8}, let $\Phi$ be the congruence $\varepsilon$ subgroup of $\mathrm{O}'(f,\mathfrak{o}_K)$ 
and let $\Delta$ be the subgroup of $\mathrm{O}'(f,\realnos)$ generated by $D$ and $\Phi$. 
Then $\Delta$ is commensurable to $\mathrm{O}'(f,\mathfrak{o}_K)$. 
Hence $\Gamma = W\Delta W^{-1}$ is an arithmetic group of isometries of $H^3$ of the simplest type defined over $K$,  
and $\gamma = WDW^{-1}$ is an orientation preserving loxodromic hyperbolic element of $\Gamma$ such that 
$\lambda^{\frac{1}{2}} = e^{\ell(\gamma)}$. 

Likewise for $L = \rationalnos(\sqrt{\beta})$, we derive a similar conclusion with 
$$D = \frac{1}{5\sqrt{\beta}}\left(\begin{matrix}
4+\sqrt{6}&1-\sqrt{6}&-1+\sqrt{6}&-4-\sqrt{6}\\
7+3\sqrt{6}&8+2\sqrt{6}&-3-2\sqrt{6}&13+7\sqrt{6}\\
-3-2\sqrt{6}&8+2\sqrt{6}&7+3\sqrt{6}&-7-3\sqrt{6}\\
-1+\sqrt{6}&1-\sqrt{6}&4+\sqrt{6}&21+9\sqrt{6}
\end{matrix}\right).$$

\end{document}